\documentclass[11pt]{article}
\date{}
\usepackage[utf8]{inputenc} 
\usepackage[T1]{fontenc}    
\usepackage[colorlinks=true,linkcolor=blue,allcolors=blue]{hyperref}       
\usepackage{url}            
\usepackage{booktabs}       
\usepackage{amsfonts}       
\usepackage{nicefrac}       
\usepackage{microtype,nicefrac}      
\usepackage{xspace}
\usepackage{natbib}
\usepackage[toc]{appendix}
\usepackage{minitoc}
\usepackage{subcaption}
\usepackage{physics}
\usepackage{comment}
\usepackage{amsmath,amsthm}
\usepackage{etoolbox}

\usepackage{etoolbox}

\AtBeginEnvironment{lemma}{\vspace{6pt}}
\AtBeginEnvironment{definition}{\vspace{6pt}}
\AtBeginEnvironment{theorem}{\vspace{6pt}}
\makeatother
\usepackage{graphicx}
\usepackage{rotating}

\usepackage{amssymb}
\usepackage{autobreak}
\usepackage{mathtools}
\usepackage[dvipsnames]{xcolor}
\usepackage{bbm}
\usepackage[ruled,vlined, linesnumbered]{algorithm2e}
\usepackage{algorithmic}
\usepackage[parfill]{parskip}

\usepackage[toc]{appendix}
\usepackage{minitoc}

\makeatletter

\newcommand{\cc}{\mathcal{C}}

\newcommand{\cz}{\mathcal{Z}}

\makeatletter
\def\mathcolor#1#{\@mathcolor{#1}}
\def\@mathcolor#1#2#3{%
  \protect\leavevmode
  \begingroup
    \color#1{#2}#3%
  \endgroup
}
\makeatother

\newcommand*{\vsepfbox}[1]{%
  \begingroup
    \sbox0{\fbox{#1}}%
    \setlength{\fboxrule}{0pt}%
    \mbox{\kern-\fboxsep\fbox{\unhbox0}\kern-\fboxsep}%
  \endgroup
}

\newtheorem{theorem}{Theorem}[section]
\numberwithin{theorem}{section}
\newtheorem{corollary}[theorem]{Corollary}

\newtheorem{lemma}[theorem]{Lemma}

\newtheorem{definition}[theorem]{Definition}

\newtheorem{example}[theorem]{Example}

\makeatletter
\newcommand{\subalign}[1]{%
  \vcenter{%
    \Let@ \restore@math@cr \default@tag
    \baselineskip\fontdimen10 \scriptfont\tw@
    \advance\baselineskip\fontdimen12 \scriptfont\tw@
    \lineskip\thr@@\fontdimen8 \scriptfont\thr@@
    \lineskiplimit\lineskip
    \ialign{\hfil$\m@th\scriptstyle##$&$\m@th\scriptstyle{}##$\hfil\crcr
      #1\crcr
    }%
  }%
}
\makeatother

\newcommand{\R}{\mathbb{R}}

\newcommand{\ang}[1]{\left\langle #1\right\rangle}

\title{Mirror-Free Proximal Methods}

\author{%
Abhijeet Vyas\\
  Purdue University\\
  \texttt{vyas26@purdue.edu} 
\and
Brian Bullins\\
  Purdue University\\
  \texttt{bbullins@purdue.edu} \\
}

\begin{document}


\maketitle

\begin{abstract}
  We present a \emph{mirror-free} mirror prox (MFMP) algorithm, which extends the classic approach of Nemirovski (2004) to allow for proximal-like updates without the explicit need for a mirror map. We further analyze the convergence of our method under suitable notions of relative smoothness and relative Lipschitzness, for which we introduce a relaxation of the standard Bregman divergence in terms of more general potential operators. Finally, we show how a strongly monotone variant of our method allows us to solve regularized Taylor-expansion subproblems that appear in both second- and third-order smooth min-max optimization.

\end{abstract}

\section{Introduction}
The framework of variational inequalities (VIs) serves as a versatile tool for expressing and analyzing problems arising in fields as varied as network flow optimization \citep{smith1979existence}, economic equilibrium modeling \citep{dafermos1984supply}, and adversarial machine learning \citep{madry2018towards}. Efficiently solving VIs, especially in non-smooth or non-monotone settings, remains a significant challenge in optimization.

In many practical applications, exact solutions may be unattainable, making it necessary to seek approximate solutions. For instance, one may aim to find, for an operator $F$, an approximate solution $z^*$ to the VI objective such that, for $\epsilon > 0$,
\begin{align}\label{eq:VIobjapprox}
    \ang{F(z), z^* - z} \leq \epsilon \quad \forall z \in \mathcal{Z}.
\end{align}
A notable special case of the VI framework is min-max optimization, which models saddle-point problems involving two agents with competing objectives. This formulation has widespread applications in areas such as game theory and bilevel optimization, among others.
Min-max optimization problems, described by the following objective:
\begin{equation}\label{eq:minmax}
\min_{x \in \mathcal{X}} \max_{y \in \mathcal{Y}} f(x,y),
\end{equation}
where $f: \mathcal{X} \times \mathcal{Y} \to \R$ is defined over convex sets $\mathcal{X} \subseteq \R^m$ and $\mathcal{Y} \subseteq \R^n$, are closely related to VIs. For the min-max setting, the operator $F$ is defined as $F(x, y) = (\nabla_x f(x, y), -\nabla_y f(x, y))$, and the VI formulation \eqref{eq:VIobjapprox} can capture the saddle-point problem.

Proximal methods, such as mirror descent \citep{nemirovski1983problem} and its generalizations like mirror prox \citep{nemirovski2004prox}, have been successful in solving VIs under various conditions. Mirror descent obtains convergence guarantees under the strict MVI condition \citep{mertikopoulos2018optimistic}, while mirror prox achieves $O(K^{-1})$ rates for objectives like operator norm and VI objective in monotone settings \citep{diakonikolas2021efficient}. Recent advances \citep{adil2022optimal, lin2022perseus, vyas2023beyond, chen2025solving} have extended these results to higher-order methods and structured non-monotone settings. However, central to these algorithms is \emph{their dependence on a particular mirror map}. 

We therefore address this critical gap in our work, whose key contributions are as follows:\vspace{-0.2cm}
\begin{itemize}
    \item We develop and analyze our novel mirror-free mirror prox (MFMP) and MFMP strongly monotone (MFMP-SM) algorithms.\vspace{-0.2cm}
    \item In establishing and analyzing our techniques, we introduce several new notions, including operator-based generalizations of smoothness and Bregman divergence, which may be of independent interest.\vspace{-0.2cm}
    \item We additionally demonstrate the broader applicability of our algorithms for effectively computing second- and third-order updates that arise in min-max optimization algorithms by showing how they satisfy appropriate operator-based relative smoothness and relative strong monotonicity conditions.\vspace{-0.2cm}
\end{itemize}

\subsection{Applications}

Smoothness (of various orders) is a key property that can ensure convergence of suitable optimization methods. While in convex optimization smoothness is traditionally defined with respect to a particular norm, recent works \citep{birnbaum2011distributed, lu2018relatively, hanzely2021accelerated} have explored a notion of relative smoothness with respect to a more general mirror map. In the context of min-max optimization,~\cite{cohen2020relative} showed that relative smoothness of a function $f$ with respect to a mirror map $\phi$ implies relative Lipschitzness of the operator $F$ with respect to $\phi$.

In this work, we extend the concept of relative smoothness and relative Lipschitzness \citep{cohen2020relative} to operators, and these generalizations allows us to analyze MFMP, which does not rely on an explicit mirror map. MFMP achieves sub-linear convergence rates for solutions to \eqref{eq:VIobjapprox} under assumptions of relative smoothness or relative Lipschitzness, thus allowing us to tackle a larger class of problems compared to those that satisfied previously considered conditions. 

Challenges in this more general case include an additional term that involves the closed-loop integral of the operator $F$. To address this, we draw inspiration from the \emph{Helmholtz decomposition} \citep{stokes1849dynamical,helmholtz1858integrale}, which allows us to break down an operator into its conservative and non-conservative components. This leads to the concept of \emph{co-conservative operators}, defined as operators that share their non-conservative components. As a result, our analysis of MFMP hinges on both the relative smoothness and co-conservative relationships between the objective operator $F$ and the mirror operator $H$. Beyond solving problems that satisfy generalized smoothness and relative Lipschitzness notions, our algorithm also allows us to solve the sub-problems associated with various high-order algorithms as discussed below. 

Higher-order derivative information about a function can be used to accelerate optimization of higher-order smooth functions in both minimization \cite{nesterov2006cubic} and min-max optimization \citep{adil2022optimal,lin2022continuous,vyas2023beyond}. However, the updates for $p^{th}$ order methods beyond $p=2$ are non-trivial to compute. For minimization tasks,~\cite{nesterov2021implementable} used the properties of convex functions to reduce finding the third-order update to solving a small number of more manageable subproblems. In the case of min-max optimization, we show how a \emph{strongly monotone} version of our MFMP method, MFMP-SM, can similarly facilitate a more effective means of approximately solving the appropriate third-order subproblems~\citep{adil2022optimal}.

Overall, this paper generalizes the notions of relative Lipschitzness, relative smoothness, and the Bregmann divergence to their operator variants, thereby relieving the reliance of their definitions on the mirror maps. Equipped with these notions we develop the MFMP algorithm which allows us to approximately solve the VI objective in Eq.~\eqref{eq:VIobjapprox} for monotone operators and achieves a sub-linear rate of the error in \eqref{eq:VIobjapprox} under the assumptions of either operator relative smoothness or operator relative Lipschitzness. Under the relative smoothness and strong monotonicity assumptions, the strongly monotone variant of the algorithm, MFMP-SM, may achieve a linear rate on the approximation error in Eq.~\eqref{eq:VIobjapprox}.

\subsection{Related Work}
Methods based on general mirror maps such as mirror descent \citep{nemirovski1983problem, beck2003mirror} and mirror prox \citep{nemirovski2004prox, tseng2008accelerated} have proven to be immensely useful tools in optimization, with applications in online learning \citep{hazan2016introduction} and information geometry \citep{raskutti2015information, amari2016information}, as well as for classic problems such as maximum flow \citep{christiano2011electrical, sherman2017area}, optimal transport \citep{jambulapati2019direct, lin2022efficiency}, and matrix scaling \citep{allen2017much, cohen2017matrix}. Accordingly, it is common in many such instances to assume that the mirror map is strongly convex \emph{with respect to a norm}, and it is this same norm that is used to define---in the case of mirror descent---smoothness of the function, or---in the case of mirror prox---Lipschitzness of the operator. In this way, the specified norm acts as an ``intermediary'' of sorts between the mirror map and the function/operator.

In an effort to bypass this norm-dependent relation, several works have considered more general \emph{relative} notions of smoothness \citep{birnbaum2011distributed, lu2018relatively,hanzely2021accelerated}, which have proven useful in such contexts as expectation maximization (EM) \citep{aubin2022mirror} and high-order tensor method implementations \citep{nesterov2021implementable}. These ideas have since been generalized to notions of relative Lipschitzness in the variational inequalities setting \citep{cohen2020relative}, whereby the latter establishes additional connections with the techniques of area-convex regularization \citep{sherman2017area, jambulapati2023revisiting}.

Additionally, there have been efforts to dispense with the mirror maps altogether. Notable in this regard is the work by \cite{gunasekar2021mirrorless} on a \emph{mirrorless} variant of mirror descent, whereby the authors consider a more general Riemannian gradient flow as the infinitesimal limit of mirror descent \citep{raskutti2015information}, where the metric tensor is taken to be the Hessian of the mirror map. As a natural consequence, this perspective allows for the consideration of flows over manifolds with more general metric tensors, that is, which may \emph{not} be the Hessian of any function, in which case there is no ``mapping'' to connect the primal and dual spaces.

\section{Preliminaries}\label{sec:prelims}
We begin by establishing key notation, definitions and supporting lemmas that will be used throughout the paper.

\begin{definition}[Bregman Divergence]
We let $\omega_\phi(z_b,z_a)$ denote the Bregman divergence with respect to a convex function $\phi$, which is defined as
\begin{equation*}
\omega_\phi(z_b,z_a) =  \phi(z_b)-\phi(z_a)-\langle \nabla \phi(z_a),z_b-z_a\rangle,
\end{equation*} where $z_a,z_b\in \mathcal{Z}$.
\end{definition}
For differentiable $\phi$, $\phi(z_b)-\phi(z_a)$ (and thus $\omega_\phi$) can be represented by an integral of the gradient of $\phi$ over any path from $z_a$ to $z_b$ (\cite{azoury2001relative}). Rather than depending on a mirror map $\phi$, the generalized Bregman divergence is defined with respect to an operator $H$. This generalization allows us to handle cases where the operator $H$ is not the gradient of any mirror map. To accommodate this difference, the GBD includes a line integral of the operator $H$, which we define as follows.
\begin{definition}[Line integral of an operator]\label{def:lintegral}
The line integral of an operator $H$ over a curve $\cc$ is as follows:
    $$\int_{\cc} H(r)^\top dr = \int_a^b \ang{H(r(t)), \dv{r(t)}{t}} dt$$
where $r : [a,b] \rightarrow \cc$ is a one-to-one and onto parameterization of the curve $\cc$ which lies in $\mathcal{Z}$ such that $r(a),~r(b)\in \mathcal{Z}$ are the starting and ending points of the curve $\cc$.  $\dv{(r(t))}{t}$ is the element wise derivative of $r(t)$ with respect to $t$. 
\end{definition}
In this work we assume that the path $P_{z_az_b}$ is a straight line from $z_a$ to $z_b$ in the Euclidean space and represent the integral of an operator $H$ over the same as $\int_{P_{z_az_b}} H(r)^\top dr = \int_{z_a}^{z_b} H(r)^\top dr$ which is defined below.
Note that over a straight line starting from $z_a$ to $z_b$, using Definition \ref{def:lintegral} the line integral of an operator can be written as
\begin{align*}
    \int_{z_a}^{z_b} H(r)^\top \, dr 
    &= \int_0^1 \ang{H(z_a + t(z_b - z_a)), z_b - z_a} \, dt~ \forall z_a, z_b \in \cz
\end{align*}
When the path $\cc$ is closed, the line integral is denoted by $\oint$. Throughout the paper we will use $P_{z_az_bz_c}$ to denote a closed path consisting of straight lines through the points $z_a,z_b,z_c\in\mathcal{Z}$, in that order, and the line integral of an operator $H$ over the path as $\oint_{P_{z_az_bz_c}} H(r)^\top dr$.
We now define the \emph{generalized} Bregman divergence.
\begin{definition}[Generalized Bregman Divergence]\label{def:gbd}
For any two points $z_a,z_b\in \mathcal{Z}$, we define the \emph{generalized Bregman divergence (GBD)} with respect to an operator $H$ as follows
\begin{equation*}
\omega_H(z_b,z_a) =  \int_{z_a}^{z_b} H(r)^\top dr- \langle  H(z_a),z_b-z_a\rangle.
\end{equation*}
\end{definition}
We prove an analogous Bregman three-point property for this generalized divergence. 
\begin{lemma}[Three point property]\label{lem:3-point}
    For any three points $z_a,z_b, z_c \in \mathcal{Z}$, the GBD with respect to $H$ satisfies
\begin{align*}
    \omega_H(z_a, z_c) + \omega_H(z_c, z_b) 
    &= \oint_{P_{z_a z_b z_c}} H(r)^\top \, dr + \omega_H(z_a, z_b)+ \ang{H(z_b) - H(z_c), z_a - z_c}.
\end{align*}
\end{lemma}
We define the monotonicity of operators and show that the GBD with respect to a monotone operator is non-negative. 
\begin{definition}[Monotonicity]\label{def:monotone}
An operator $F$ is monotone if for all $z_a,z_b \in \cz$,
\[\ang{F(z_b)-F(z_a),z_b-z_a} \geq 0.\]
\end{definition}

Equivalently, we have the following.
\begin{lemma}\label{lem:jacobmon}
A differentiable operator $F$ is monotone if and only if for all $h\in \cz$ its Jacobian satisfies
\begin{equation}\label{def:jacobismooth}
    h^\top \nabla F(z)h  \geq 0.
\end{equation}
\end{lemma}
It follows that the GBD with respect to a monotone operator is non-negative.
\begin{lemma}\label{lem:posgbd}
    If the operator $F$ is monotone, we have for all $z_a,z_b\in \cz$,
        $$\omega_F(z_b,z_a)\geq 0.$$ 
\end{lemma}  
We now present key relations between operators $F$ and $H$ under which we analyze its convergence, beginning with a notion of \emph{operator} relative smoothness.  
\begin{definition}[Operator Relative Smoothness]
We define an operator $F$ to be $L$-relatively smooth with respect to another operator $H$ if for all $z_a, z_b \in \cz$, 
\begin{align}\label{def:smooth}
    L \ang{H(z_b) - H(z_a), z_b - z_a}
    &\geq \ang{F(z_b) - F(z_a), z_b - z_a}
\end{align}
\end{definition}
We note that our notion of \emph{operator} relative smoothness is inspired by the relative smoothness condition of~\cite{lu2018relatively}, whereby ours generalizes their notion by letting $F = \nabla f$ and $H = \nabla \phi$ be the gradients of relatively smooth functions $f$ and $\phi$, respectively. Next, letting $\nabla F$ denote the Jacobian of an operator $F$, we prove an equivalent definition in the following lemma.
\begin{lemma}[Jacobians of relatively smooth operators]\label{lem:jacobrels}
An operator $F$ is $L$-relatively smooth with respect to another operator $H$ if and only if their Jacobians satisfy for any $z,h\in \mathcal{Z}$,
\begin{equation}\label{def:jacobirelsmooth}
    L h^\top \nabla H(z) h \geq h^\top \nabla F(z) h.
\end{equation}

\end{lemma}\vspace{-12pt}
We now relate the GBDs between two relatively smooth operators.
\begin{lemma}\label{lem:gbdrelopsmooth}
If an operator $F$ is $L$-relatively smooth with respect to another operator $H$, then we have for all $z_a,z_b \in \cz$, 
\begin{equation}
    L \omega_H(z_b,z_a) \geq \omega_F(z_b,z_a).
\end{equation}
\end{lemma}

\begin{lemma}\label{lem:relopsmooth}
    If a monotone operator $F$ is $L$-relatively smooth with respect to operator $H$, then we have for any three points $z_a, z_b, z_c \in \cz$, 
\begin{align*}
    L(\omega_H(z_b, z_c) + \omega_H(z_a, z_b)) 
    &\geq \oint_{P_{z_a z_b z_c}} F(r)^\top \, dr \nonumber+ \ang{F(z_c) - F(z_b), z_a - z_b} 
\end{align*}
\end{lemma}
\begin{definition}[Operator Relative Lipschitzness]\label{def:reloplip}
    An operator $F$ is $L$-relatively Lipschitz with respect to operator $H$ if for any three points $z_a, z_b, z_c \in \mathcal{Z}$,
\begin{align*}
    L(\omega_H(z_b,z_c)+\omega_H(z_a,z_b)) \geq  \ang{F(z_c)-F(z_b),z_a-z_b}.
\end{align*} 
\end{definition}\vspace{-12pt}
Our notion of operator relative Lipschitzness was inspired by the relative Lipschitzness condition of \cite{cohen2020relative} whereby ours generalizes their notion with $H(z) = \nabla \phi(z)$, in which case the GBD, $\omega_H$ is equivalent to $\omega_{\phi}$.

We further note that in the convex setting if a function $f$ is $L$-relatively smooth with respect to a function $h$, the operator $F = \nabla f$ satisfies the relative Lipschitzness condition with respect to the Bregman divergence $\omega_h$ with mirror map $h$ \citep{cohen2020relative}. However this implication does not generalize for the min-max operator $F(x,y) = (\nabla_x f(x,y), -\nabla_y f(x,y))$ in the monotone setting. This is because the min-max operator is not the gradient of any function and we obtain an extra term $\oint_{P_{z_az_bz_c}}F^\top dr$ in the expression of Lemma \ref{lem:relopsmooth} (which assumes relative smoothness) which is absent in the definition \ref{def:reloplip} of relative Lipschitzness. Thus, for operators that are a gradient of a function (and therefore \textit{conservative}), this term would be zero for all closed paths $P_{z_az_bz_c}$ and the definitions of relative-smoothness and relative-Lipschitzness would coincide.

The fundamental theorem of vector calculus \citep{stokes1849dynamical,helmholtz1858integrale} states that any sufficiently continuous operator in three dimensions can be written as the sum of a conservative and non-conservative operator via the Helmholtz decomposition. The line-integral of any operator over a closed curve $\cc$ would then simply be the line-integral of its non-conservative part. In light of this discussion we introduce the notion of $\delta$ \emph{conservative} and $\delta$ \emph{co-conservative} operators.

\begin{definition}[Conservative Operators]\label{def:conservativeop}
The operator $H$ is a $\delta$ conservative operator if, for every curve $\cc$, $H$ satisfies $-\delta \leq \oint_{\cc} H (r)^\top dr \leq \delta.$
\end{definition}

\begin{definition}[Co-conservative operators]\label{def:coconservative}
Let $\Delta$ be the difference operator $\Delta(x,y) = E(x,y)-F(x,y)$ where $E$ and $F$ are operators maps from $\cz \rightarrow \mathbb{R}^d$. The operator $F$ is $\delta$ co-conservative with respect to $E$ if for every closed curve $\cc$,
$-\delta \leq \oint_{\cc} \Delta (r)^\top dr \leq \delta.$
\end{definition}

Note that Definition \ref{def:conservativeop} provides a bound on the non-conservative part of the operator $H$, while Definition \ref{def:coconservative} provides a bound on the non-conservative part of the operator $\Delta$, which would be 0 if the operators $F$ and $H$ shared the same non-conservative parts. Analogous to the relative smoothness of an operator $F$ defined with respect to an operator $H$, we can also define the notion of relative strong monotonicity of an operator $F$ relative to an operator $H$.

\begin{definition}[Operator relative strong-monotonicity]\label{def:stronglymono}
We define an operator $F$ to be $m$-relatively strongly monotone with respect to another operator $H$ if for all $z_a, z_b \in \cz$, 
\begin{align*}
    \ang{F(z_b) - F(z_a), z_b - z_a} 
    &\geq m \ang{H(z_b) - H(z_a), z_b - z_a}.
\end{align*}
\end{definition}\vspace{-12pt}
This notion generalizes the relative strong-monotonicity of an operator $F$ (with respect to a function $h$), as introduced by~\cite{cohen2020relative}, and, as we later discuss in Section~\ref{sec:algorithms} it allows us to achieve linear rate under operator relative smoothness. The proofs of the statements in this section are provided in Appendix \ref{app:sec2}.\vspace{-12pt}
\section{Mirror-Free Algorithms}\label{sec:algorithms}
In this section, we present our main algorithm, \emph{Mirror-Free Mirror Prox} (MFMP), along with its convergence analysis under various conditions. At a high level, MFMP can be viewed as a generalization of the mirror-prox method in which the proximal steps are taken with respect to a \emph{general mirror operator} rather than the gradient of a mirror map. This allows the method to operate without explicitly requiring a mirror function while retaining the key structural properties that enable mirror-prox--type analyses.

For the special case in which the operator also satisfies \emph{relative strong monotonicity}, we introduce a generalized variant called MFMP-SM. Both MFMP and MFMP-SM rely on a proximal operation defined with respect to a mirror operator $H$. Intuitively, this operation plays the same role as the standard mirror-prox update: it computes a step that balances progress in the direction of the operator with a geometry induced by $H$.

We first define this proximal operation, which will be used by both MFMP and MFMP-SM (an additional operation specific to MFMP-SM will be introduced later). The update replaces the gradient of the mirror map $\nabla h$ used in classical mirror-prox methods with a general mirror operator $H$, thereby extending the class of geometries under which the algorithm can be applied.
\begin{align*}
    \textrm{Prox}_H(z_a, z_b)= \{ z' : \ang{F(z_b) + L(H(z') - H(z_a)), z' - z} \leq 0 \quad \forall z \}.
\end{align*}
The MFMP algorithm performs two such proximal updates at each iteration: one using the mirror operator evaluated at $z_k$, and another evaluated at the intermediate point $z_{k+\frac{1}{2}}$. This mirrors the structure of the classical mirror-prox algorithm, but within the more general operator-based framework described above.

\noindent
\begin{algorithm}[H]
  \KwIn{Initial point $z_1 \in \mathcal{Z}$, operators $F$ and $H$ such that $F$ is $L$-relatively smooth with respect to $H$.}
  \For{$k = 1$ \KwTo $K$}{
        $z_{k+\frac{1}{2}} \in \textrm{Prox}_H(z_k,z_{k})$  \\
        $z_{k+1} \in \textrm{Prox}_H(z_k,z_{k+\frac{1}{2}})$  \\
  }
  \Return{$z_{\mathrm{out}} = \frac{1}{K}\sum_{k=1}^K z_{k+\frac{1}{2}}$}
  \caption{Mirror-Free Mirror Prox (MFMP)}\label{ALG:1}
\end{algorithm}
\begin{algorithm}[H]
  \KwIn{Initial point $z_1 \in \mathcal{Z}$, operators $F$ and $H$ such that $F$ is $L$-relatively smooth and $m$-strongly-monotone with respect to $H$.}
  \For{$k = 1$ \KwTo $K$}{
        $z_{k+\frac{1}{2}} \in \textrm{Prox}_H(z_k,z_{k})$  \\
        $z_{k+1} \in \textrm{Prox}_H^{SM}(z_k,z_{k+\frac{1}{2}})$  \\
  }
  \Return{$z_{\mathrm{out}} = z_K$}
  \caption{Strongly Monotone Mirror-Free Mirror Prox (MFMP-SM)}\label{ALG:2}
\end{algorithm}
The following lemma provides control over the progress made at each iteration. Such bounds are standard in analyses of mirror-prox--type methods (see, e.g.,~\cite{nemirovski2004prox}) and will play a central role in establishing the main convergence results of this section.
\begin{lemma}\label{lem:mfmp}
    The iterates of Mirror-Free Mirror Prox (Algorithm~\ref{ALG:1}) satisfy, 
    for all $k \geq 1$,  $z \in \mathcal{Z}$, 
\begin{align*}
  \langle F(z_{k+\frac{1}{2}}),& z_{k+\frac{1}{2}} - z \rangle 
  \leq L \left\langle H(z_k) - H(z_{k+1}), z_{k+1} - z \right\rangle - \left\langle F(z_k) - F(z_{k + \frac{1}{2}}), z_{k+\frac{1}{2}} - z_{k+1} \right\rangle \nonumber \\
  &\quad \quad+ L \left\langle H(z_k) - H(z_{k+\frac{1}{2}}), z_{k+\frac{1}{2}} - z_{k+1} \right\rangle.
\end{align*}
\end{lemma}
We now present the convergence results of our algorithm under the assumptions discussed in Section~\ref{sec:prelims}. Following Lemma \ref{lem:mfmp}, we obtain convergence results based on additional assumptions on the operators $F$ and $H$, beginning with those determined under relative smoothness.
\begin{theorem}[Operator relative smoothness guarantees]\label{thm:1}
    Let $F$ be $L$-relatively smooth with respect to $H$ and $\delta_1$ co-conservative with respect to $LH$, and let $H$ be $\delta_2$ conservative. After $K \geq 1$ iterations, the output $z_{out}$ of the MFMP algorithm satisfies, for initialization $z_1 \in \mathcal{Z}$ and any point $z\in \mathcal{Z}$,
    \begin{align*}
        \ang{F(z),z_{out}-z} \leq \frac{L}{K} \omega_H(z,z_1) +\frac{\delta_1}{L}+3\delta_2.
    \end{align*}
\end{theorem}
This is analogous to the original mirror prox method \citep{nemirovski2004prox} that achieves an approximation error of $\epsilon = O(K^{-1})$ after $K$ iterations. The approximation error obtained by MFMP includes an error term of $\frac{\delta_1}{L}+3\delta_2$ that arises due to the non-conservative nature of the objective and mirror operators. This error term is reminiscent of the bias term in stochastic optimization \citep{ajalloeian2020convergence} that appears due to the gradient estimates obtained being biased.   
The analysis in Theorem \ref{thm:1} uses the relative smoothness between the objective and mirror operators. The error term that arises due to the non-conservative nature of the these operators is bounded using the conservative and co-conservative properties. 

As noted, analyzing relative Lipschitzness between the objective and mirror operators isolates error terms from the mirror operator's non-conservativeness. This analysis appears in the proof of Theorem~\ref{thm:2}, which shows that MFMP solves the VI objective~\eqref{eq:VIobjapprox} up to an error of order $O(K^{-1}+L\delta)$.
\begin{theorem}[Relatively Lipschitz operators]\label{thm:2}
    Let $F$ be $L$-relatively Lipschitz with respect to $H$ and let $H$ be $\delta$ conservative. After running for $K$ iterations, the output $z_{out}$ of the MFMP algorithm satisfies 
    \begin{align*}
        \ang{F(z),z_{out}-z} \leq \frac{L}{K} \omega_H(z,z_1) +2L\delta ~\forall~z\in \cz
    \end{align*}
\end{theorem}
The above rate can be improved under strong monotonicity by relying on the following proximal update:
\begin{align*}
    \textrm{Prox}_H^{SM}&(z_a,z_b)= \{ z' : \langle F(z_b) + L(H(z') - H(z_a)) + m(H(z') - H(z_b)), z' - z \rangle \leq 0 \quad \forall z \}.
\end{align*}
We now present the Strongly Monotone Mirror-Free Mirror Prox (MFMP-SM) algorithm, which generalizes the strongly monotone mirror prox algorithm of~\cite{cohen2020relative}.
\begin{theorem}\label{thm:MFMP-SM}
    The iterates $z_k,z_{k+\frac{1}{2}}$ and $z_{k+1}$ of the MFMP-SM algorithm (Algorithm \ref{ALG:2}) satisfy
\begin{align*}
     \omega_H(z^*,z_{k+1})\leq \frac{L}{m+L}\omega_H(z^*,z_{k})+\frac{E_k}{m+L}
\end{align*}
where the error term $E_k$ is
\begin{align*}\vspace{-12pt}
    E_k &= L \oint_{P_{z_{k+1}z_k z_{k+\frac{1}{2}}}} H^T(r) \, dr+ L \oint_{P_{z^*z_kz_{k+\frac{1}{2}}}} H^T(r) \, dr \\
    &+ m \oint_{P_{z^*z_{k+\frac{1}{2}}z_{k+1}}} H^T(r) \, dr - L \oint_{P_{z_kz_{k+\frac{1}{2}}z_{k+1}}} F^T(r) \, dr.\vspace{-12pt}
\end{align*}
\end{theorem}
\begin{corollary}\label{cor:final}
    The iterates of the MFMP-SM algorithm (Algorithm \ref{ALG:2}) after $K$ iterations satisfy
\begin{align*}
     \omega_H(z^*,z_{out})\leq (\frac{L}{m+L})^K\omega_H(z^*,z_{0})+\sum_{k=0}^{K-1}\frac{E_k}{(m+L)^k}. 
\end{align*}
\end{corollary}
\begin{proof}
The proof follows by repeatedly applying the result of Theorem~\ref{thm:MFMP-SM}, since for $z_{out} = z_K$ as the output of the MFMP-SM algorithm after $K$ iterations, we have
\begin{align*}
    \omega_H(z^*,z_{K})&\leq \frac{L}{m+L}\omega_H(z^*,z_{K-1})+\frac{E_{K-1}}{m+L}\\
    &\leq (\frac{L}{m+L})^2\omega_H(z^*,z_{K-2})+\frac{E_{K-1}}{(m+L)^2}+\frac{E_{K-2}}{m+L}\\
    &\dots\\
    &\leq (\frac{L}{m+L})^K\omega_H(z^*,z_{0})+\sum_{k=0}^{K-1}\frac{E_k}{(m+L)^k}.\qedhere
\end{align*}
\end{proof}
Thus MFMP-SM may achieve a linear rate of convergence of the Bregman divergence of the iterates from the solution, $\omega_H(z^*,z_{k})$, up to an error term. The error term $E_k$ consists of loop integrals of the operators $F$ and $H$ over three points chosen from $z_k,z_{k+\frac{1}{2}},z_{k+1}$ and $z^*$. Note that in the case where the operators $F$ or $H$ are conservative, the corresponding loop integrals vanish. Additionally, the loop terms involving $F$ may sum to zero due to the disjoint nature of the terms. The proofs of the statements in this section are provided in Appendix~\ref{app:sec3}.
\section{Third-order min-max optimization}\label{sec:thirdorder}
Several recent works have focused on the use of higher-order methods to solve the VI objective \eqref{eq:VIobjapprox}. While implementing the first- and second-order instances of these methods has been addressed~\citep{lin2022explicit}, even higher-order instances provide additional challenges. In this section, we present an important pair of relatively smooth and strongly monotone operators which allows us to implement the third-order step in the higher-order algorithms used to solve Eq.~\eqref{eq:minmax}. 
The HOMVI algorithm and its variants \citep{adil2022optimal,lin2022perseus,vyas2023beyond} solve the VI objective \eqref{eq:VIobjapprox} for a monotone operator $\Phi$.
The third-order step of HOMVI uses a model operator $F$, for which we show how to construct a relative operator $H$.  The operator $\Phi$ is assumed to be third-order smooth. We will define the notion of $p^{th}$-order smoothness and the $p^{th}$-order Taylor operator (provided in Appendix \ref{app:sec4}) for which we define directional derivatives. 
\begin{definition}[Directional derivative]
    Consider a $k$-times differentiable operator $F:\mathcal{Z} \rightarrow \mathbb{R}^d$, and let $z \in \mathcal{Z}$, $\mathcal{Z}\subseteq \mathbb{R}^d$, $h \in \mathbb{R}^d$. For $r<k+1$, we let
    $$\nabla^k F(z)[h]^r = \frac{\delta^k
    }{\delta h}|_{t_1=0,\dots ,t_r=0}F(z+t_1h+\dots +t_rh)$$
    denote the $k^{th}$ directional derivative of $F$ at $z$ along $h$.
\end{definition}
For a third-order smooth operator $\Phi$, the third-order sub-problem is defined as follows.
\begin{definition}[Third-order sub-problem]\label{def:third_order_sub}
    \textit{Find} $z_b$ such that
    \begin{align*}
        & \ang{\Phi(z_a) + \nabla_x \Phi(z_a) [h] + \frac{1}{2} \nabla^2 \Phi(z_a)[h, h]+2M \nabla_h d_4(h), h+z_a - z} \leq 0\quad \forall ~z \in \cz
    \end{align*}
where $h=z_b-z_a$.
\end{definition}

We first present an operator $H$ such that $F$ is relatively smooth and strongly monotone with respect to $H$, thus allowing us to solve the sub-problem. Furthermore, we present a subclass of functions $f$ in Eq.~\eqref{eq:minmax} and their corresponding model operators $F$ for which a \emph{conservative} relative operator $H'$ exists, allowing low residual error for MFMP-SM.
\begin{lemma}\label{lem:thirdrel}
    Consider the operators
    \begin{align*}
        F(h) &= \Phi(z_a) + \nabla \Phi(z_a)[h] + \frac{1}{2} \nabla^2 \Phi(z_a)[h,h] + 2M \nabla_h d_4(h), \\
        H(h) &= \frac{1}{2} \left( 1 - \frac{1}{\tau} \right) \nabla_z \Phi(z_a)[h] + \frac{M - \tau L_3}{2} \nabla_h d_4(h),
    \end{align*}
    where $\tau > 0$. The operator $F$ is 1-strongly monotone and $\frac{\tau+1}{\tau-1}$ relatively smooth with respect to $H$.
\end{lemma}

\begin{lemma}[Conservative relative operator]\label{lem:conrelop}
    All functions of the form $f(x,y) = \alpha(x)-\beta(y)+x^\top Ay$, where $\alpha(x)$ and $\beta(y)$ are convex functions in $x$ and $y$ respectively where $z = (x,y) \in \mathcal{Z}$ correspond to operators $\Phi(x,y) = (\nabla_x \alpha(x)+Ay,\nabla_y \beta(y)-A^\top x)$, which give rise to model operators $F = \Phi(z_a)+\nabla \Phi(z_a)[h]+ \nabla^2 \Phi(z_a) [h,h]+\frac{M}{2}\nabla_h d_4(h)$ that are $\frac{\tau+1}{\tau-1}$ relatively smooth and $1$-strongly monotone with respect to \textbf{conservative} operators $H'$ of the form
\begin{align*}
    H'(h) &= \left( \nabla_x \alpha(x_a) + A y_a, \nabla_y \beta(y_a) - A^\top x_a \right) + \frac{1 - \frac{1}{\tau}}{2} \begin{bmatrix}
        \nabla_x^2 \alpha(x_a) & 0 \\
        0 & \nabla_y^2 \beta(y_a)
    \end{bmatrix} h\\
    &+ \frac{M - \tau L_3}{2} \nabla_h d_4(h)\ .
\end{align*}
\end{lemma}
We note that, using similar techniques as in~\cite{lin2022explicit}, the mirror-free updates (involving $\text{Prox}^H$ and $\text{Prox}_{SM}^H$) can be shown to converge super-linearly when used for solving the third-order sub-problem described above. Additional discussion and analysis is provided in Appendix~\ref{app:updates}. We now present another example of a pair of relatively smooth and strongly monotone operators for which the MFMP-SM algorithm can be utilized to perform a higher-order update. We consider the competitive gradient optimization (CGO) algorithm \citep{vyas2023competitive} (which generalizes \citep{schafer2019competitive}). To obtain the next iterate point $z_b$ at any iteration point $z_a$, CGO finds $z_{b}=z_a+h$ such that $\Phi_\alpha(h,z_a) = F(z_a)+\frac{\alpha}{\eta}(\nabla_{xy} f(x,y)|_{(x,y)=z_a} h_y,-\nabla_{yx} f(x,y)|_{(x,y)=z_a} h_x)+\frac{1}{\eta}h=0$ where $F(z) = (\nabla_x f(x,y),-\nabla_y f(x,y))$.
\begin{lemma}\label{lem:cgo}
         The operator $\Phi_\alpha(h,z_a)$ is 1-relatively smooth and 1-strongly monotone with respect to $\Phi_0(h,z_a) = F(z_a)+\frac{1}{\eta} h$. Furthermore we have that $\Phi_0$ is monotone.
\end{lemma}\vspace{-6pt}
This fact allows us to solve the higher order CGO update by solving for the first-order operator $\Phi_0$. The proofs of the statements in this section are provided in Appendix \ref{app:sec4}.
\section{Examples and Analysis}\label{sec:exanples}

In this section, we discuss pairs of relatively-smooth and strongly-monotone operators, in addition to presenting an analysis of the relatively-Lipschitz condition.

\subsection{Examples of relatively smooth operators}

\begin{example}\label{eg:smooth}
    Consider the operators $F : \mathbb{R}^{2n} \rightarrow \mathbb{R}^{2n}$ and $H : \mathbb{R}^{2n} \rightarrow \mathbb{R}^{2n}$,
    \begin{align*}
        F(x,y) = \left( \nabla_x f_s(x,y), -\nabla_y f_s(x,y) \right),~~~
        H(x,y) = \left( \nabla_x h_s(x,y), -\nabla_y h_s(x,y) \right)
    \end{align*}
    where $f_s(x,y) = f(x) - f(y) + x^\top B y$ and $h_s(x,y) = h(x) - h(y)$, for
    \begin{align*}
        f(x) = \frac{1}{4}\|Ex\|_2^4 + \frac{1}{4}\|Ax - b\|_4^4 + \frac{1}{2}\|Cx - d\|_2^2,~~~~~h(x) = \frac{1}{4}\|x\|_2^4 + \frac{1}{2}\|x\|_2^2,
    \end{align*}

    where $A,B,C,E \in \mathbb{R}^{n\times n}$ and $b,d\in \mathbb{R}^n$. Then, the operator $F$ is L-relatively smooth and $m$-relatively strongly monotone with respect to an operator $H: \mathbb{R}^{2n} \rightarrow \mathbb{R}^{2n}$, $H(x,y) = (\|x\|^2 x + x, \|y\|^2 y + y)$
    with $m = \min \left\{ \frac{\sigma_E^4}{3}, \sigma_C^2 \right\}$ and $L = 3\|E\|^4 + 3\|A\|^4 + 6\|A\|^3\|b\|_2^2 + 3\|A\|^2\|b\|_2^2 + \|C\|^2$ and where $\sigma_E = \min(\Lambda(E^\top E))$ and $\sigma_C = \min(\Lambda(C^\top C))$, and $\Lambda(M)$ denotes the set of eigenvalues of the matrix $M$.
\end{example}

Note that setting \( B = 0 \) makes the problem separable, and because both operators \( F \) and \( H \) are involved, the error term vanishes, leading to \( E = 0 \), as shown in Corollary \ref{cor:final}. Furthermore, in this scenario, the relative operator \( H \) is conservative since it represents the gradient of the function \( h_m = h(x) + h(y) \). Thus, from Corollary~\ref{cor:final}, the error terms related to the loop integrals of \( H \) vanish, resulting in $E = \sum_{k=1}^K L \oint_{P_{z_k z_{k + \frac{1}{2}} z_{k+1}}} F^T(r) \, dr.
$ It is known that in the separable case, the operator \( F \) is smooth and strongly monotone with respect to \( H \). Additionally, since \( F \) is conservative, it satisfies the relative Lipschitz condition for operators. This implies that the problem can also be solved by running the Mirror Prox-SM (Algorithm 3) from \cite{cohen2020relative}.

\begin{example}\label{eg:eg2}
 Upon running MFMP-SM (Algorithm \ref{ALG:2}) on the third-order sub-problem (\ref{def:third_order_sub}) for the operator $F(x,y) = (\nabla_x f(x,y),-\nabla_y f(x,y))$, where $f(x,y) = \|x\|^4-\|y\|^4+x^\top A y$, the last iterate $z_K = (x_K,y_K)$ satisfies,
\begin{align*}
    &\ang{F(z_K), z_K - z} \leq 2G \sqrt{\frac{\omega_H(z_K, z^*)}{\kappa(\tau, z_a)}}\leq \frac{2G}{\sqrt{\frac{1 - \tau}{2} \kappa(z_a)}}\sqrt{
        \left( \frac{L}{m + L} \right)^K \!\!\!\omega(z^*, z_0)
        + \sum_{k=1}^K \frac{E_k}{(m + L)^k}
    } 
\end{align*}
where $E_k = L\oint_{P_{z_kz_{k+1}z_{k+\frac{1}{2}}}} F^\top(r) dr$ and  $\kappa(z_a) = \min(\|x_a\|,\|y_a\|)$.
\end{example}

\subsection{Analyzing Relative Lipschitzness}
In the examples considered in the previous section, we find that the relative operator $H$ corresponding to the main operator $F$ is conservative. If the relative operator is conservative the notions of relative Lipschitzness and relative smoothness are equivalent to operator relative Lipschitzness and operator relative smoothness respectively. If the operator $F$ satisfies the relative Lipschitzness condition we can potentially use the strongly-monotone mirror prox algorithm from \cite{cohen2020relative} with the appropriate mirror map $\phi$. However, we show that neither examples satisfy the relative Lipschitzness condition.

\begin{theorem}\label{lem:antilip}
    For an operator $F$ that is $L$-relatively smooth and $m$ strongly-monotone with respect to $H$, we have
$$L-m \geq \frac{\oint_{P_{abc}}F^\top dr-\omega_F(z_a,z_c)}{\omega_H(z_a,z_b)+\omega_H(z_b,z_c)}.$$
\end{theorem}

We now proceed to analyze the aforementioned examples.
\begin{corollary}\label{cor:norelip1}
    The operator pair $F$ and $H$ in Examples \ref{eg:smooth} and \ref{eg:eg2} do not satisfy the operator relative Lipschitzness condition (Definition \ref{def:reloplip}) for all possible parameters.
\end{corollary}

This impossibility occurs due to the difference in the order of terms in the numerator and denominator (when they are parametrized with a single variable $\theta$) which arises from the non-conservative parts of the operator $F$. A figure of the path $P_{abc}$ used in the proof is provided in Figure \ref{figure:relLip}. The proofs are provided in Appendix \ref{app:sec5}.

\begin{figure}[h]
    \centering
    \includegraphics[width=50mm]{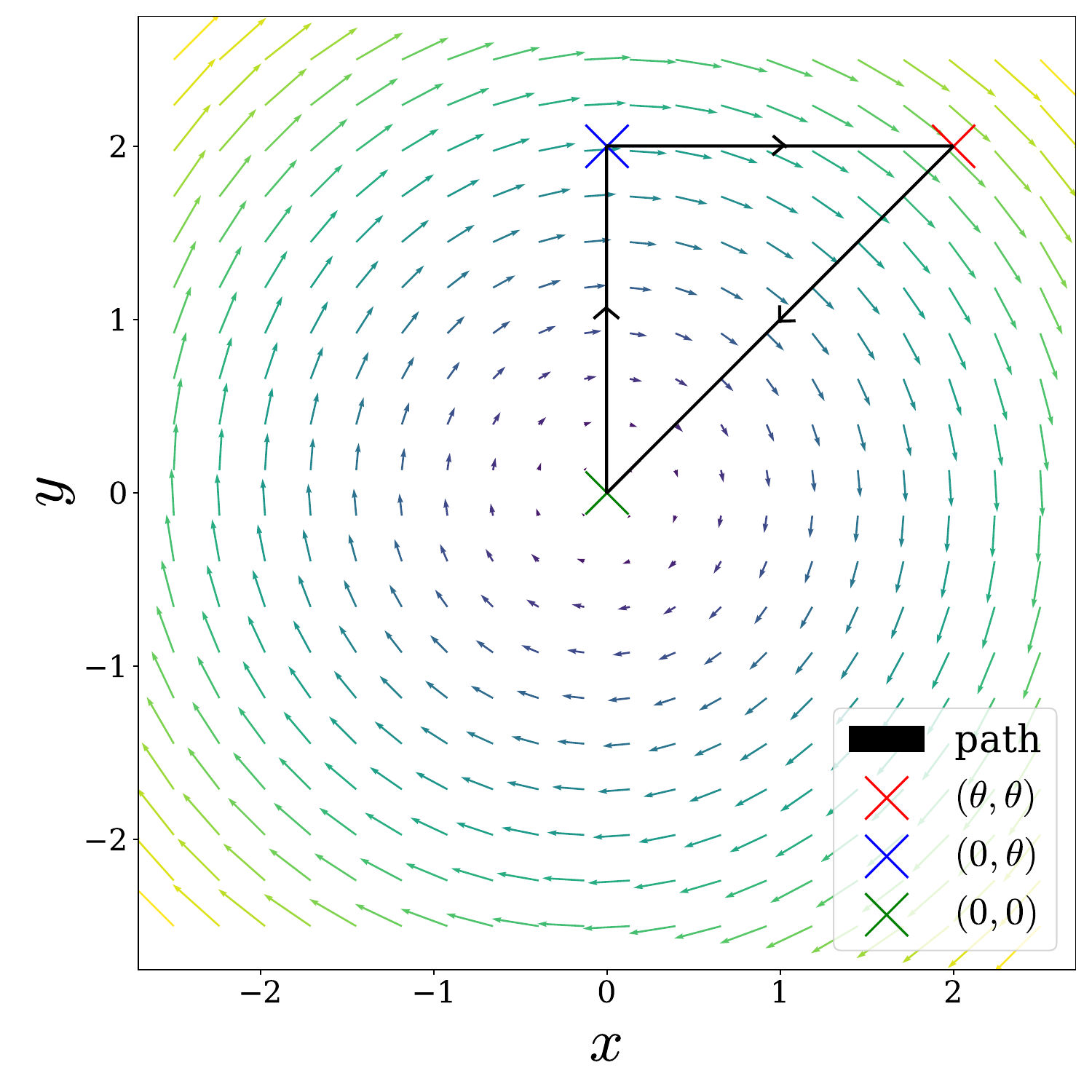}
    \caption{An illustration of the construction used to prove Corollary \ref{cor:norelip1}.}
    \label{figure:relLip}
\end{figure}
\section{Conclusion}

        We have established a \emph{mirror-free} variant of the mirror prox algorithm, based in part on a certain generalized Bregman divergence defined with respect to an operator, rather than any particular mirror map. In addition, we have analyzed our method under the conditions of relative smoothness and relative Lipschitzness, and our rates provide a natural interpolation between conservative and non-conservative operators. Analyzing the method under additional conditions, such as for weakened notions of monotonicity, could provide interesting directions for future works. Finally, we believe our techniques may provide new opportunities for exploring directions in min-max optimization and variational inequalities in settings beyond those based on more standard assumptions.

\bibliography{mirrorfree}
\bibliographystyle{plainnat}

\newpage

\newpage
\appendix
\addcontentsline{toc}{section}{Appendix}
\part{Appendix} 
\parttoc 
\section{Proofs for Section \ref{sec:prelims}}\label{app:sec2}

\textbf{Proof of Lemma \ref{lem:3-point}}
\begin{proof}
By Definition~\ref{def:gbd}, we have that
\begin{align*}
        \omega_H(z_a,z_c)+\omega_H(z_c,z_b) -\omega_H(z_a,z_b) 
        &= {\int_{z_c}^{z_a}} H^Tdr -\langle H(z_c), z_a-z_c \rangle+{\int_{z_b}^{z_c}} H^Tdr -\langle H(z_b), {z_c}-z_b \rangle\\
        &\quad\quad- {\int_{z_b}^{z_a}} H^Tdr -\langle H(z_b), z_a-{z_b} \rangle\\
        &= \oint_{P_{z_az_bz_c}} H^Tdr+ \langle H({z_b})-H({z_c}),z_a-z_c\rangle.
\end{align*} 
where the last equality holds due to the fact that $-{\int_{z_b}^{z_a}} H^Tdr  = {\int_{z_a}^{z_b}} H^Tdr$ and $P_{z_az_bz_c}$ denotes a path of three line segments, $z_a\rightarrow z_b$, $z_b\rightarrow z_c$, and $z_c\rightarrow z_a$.
\end{proof}

\textbf{Proof of Lemma \ref{lem:jacobmon}}
\begin{proof}
    From \cite{rockafellar2009variational}, Proposition 13.2, we have that the symmetric part of the Jacobian of the operator $F$ is positive semi-definite, i.e. $h^\top (\frac{\nabla F(z)+\nabla F(z)^\top}{2})h \geq 0~\forall h,z \in \cz$ if and only if the operator $F$ is monotone. But we have $h^\top (\frac{\nabla F(z)+\nabla F(z)^\top}{2})h \geq 0~\forall h,z \in \cz$ if and only if $0 \leq h^\top \nabla F(z) h ~\forall h\in \cz$ since $h^\top\left(\frac{\nabla F(z)-\nabla F(z)^\top}{2}\right) h = 0~\forall~h,z\in \cz$, and so the desired equivalence follows.
\end{proof}

\textbf{Proof of Lemma \ref{lem:posgbd}}
\begin{proof}
For any $t \geq 0$, we have that, for monotone operator $F$,
\begin{equation*}\ang{F(z_a+t(z_b-z_a))-F(z_a),t(z_b-z_a)} \geq 0 \implies \ang{F(z_a+t(z_b-z_a))-F(z_a),(z_b-z_a)} \geq 0,
    \end{equation*}
    and so it follows that
    \begin{align*}
    &\omega_F(z_b,z_a) = \int_0^1 \ang{F(z_a+t(z_b-z_a))-F(z_a),z_b-z_a} dt\geq 0.\qedhere 
    \end{align*}
\end{proof}

\textbf{Proof of Lemma \ref{lem:jacobrels}}
\begin{proof}
Since $F$ is $L$-relatively smooth with $H$ we have that the operator $\Phi(z) = LH(z)-F(z)$ is monotone. From Lemma \ref{lem:jacobrels} we have that $h^\top\nabla \Phi(z)h \geq 0~\forall~h\in \cz$. Since $\nabla \Phi(z) =  L\nabla H(z)-\nabla F(z)$ we complete the proof.
\end{proof}

\textbf{Proof of Lemma \ref{lem:gbdrelopsmooth}}
\begin{proof}
    From Eq.~\eqref{def:smooth} for $t \geq 0$, we obtain
    \begin{align*}
        \ang{F(z_a)-F(z_a+t(z_b-z_a)),z_b-z_a} \leq L \ang{H(z_a)-H(z_a+t(z_b-z_a)),z_b-z_a}.
    \end{align*}

    Integrating, we have
    \begin{align*}
        \int_0^1 \ang{F(z_a)-F(z_a+t(z_b-z_a)),z_b-z_a} dt\leq 
        L \int_0^1 \ang{H(z_a)-H(z_a+t(z_b-z_a)),z_b-z_a}dt.
    \end{align*}

    However, we note that, by Definition~\ref{def:gbd},
    \begin{align*}
    \int_0^1 \ang{F(z_a)-F(z_a+t(z_b-z_a)),(z_a-z_b)} dt 
    &=  \int_0^1 F(z_a+t(z_b-z_a))^\top(z_a-z_b)dt- \ang{F(z_a),z_b-z_a}\\ 
    &= \omega_F(z_b,z_a).
    \end{align*}

    Observing the same for the operator $H$, we have

    $$\omega_F(z_b,z_a) \leq L \omega_H(z_b,z_a) ~~ \forall z_a,z_b\in \cz,$$

    which is the statement of the lemma.
\end{proof}

\textbf{Proof of Lemma \ref{lem:relopsmooth}}

\begin{proof}
    By the generalized Bregman three-point property,
    \begin{align*}
            \omega_F(z,z_K)& = \omega_F(z^*,z_K)-\omega_F(z^*,z)+\oint_{P_{zz^*z_K}}F^Tdr+\ang{F(z_k)-F(z),z^*-z}\\
                           & \leq \oint_{P_{zz^*z_K}}F^Tdr-\int_{z}^{z^*}F^Tdr+\ang{F(z_k),z^*-z}+L\omega(z^*,z_K)
    \end{align*}

    where in the second step we expand $\omega_F(z^*,z)$ and use the fact that $\omega_F \leq L\omega$. Similarly expanding $\omega_F(z,z_K)$ we have,

    \begin{align*}
            \ang{F(z_K),z_K-z} & \leq \oint_{P_{zz^*z_K}}F^Tdr-\int_{z}^{z^*}F^Tdr-\int_{z_K}^zF^Tdr+\ang{F(z_k),z^*-z}+L\omega(z^*,z_K)\\
                               & = \int_{z^*}^{z_K}F^Tdr+\ang{F(z_K),z^*-z}+L\omega(z^*,z_K)\\
                               & = \ang{F(z^*),z_K-z^*}+\ang{F(z_K),z^*-z}+L(\omega(z^*,z_K)+\omega_H(z_K,z^*))\\
                               & \leq \ang{F(z_K),z^*-z}+L(\omega(z^*,z_K)+\omega_H(z_K,z^*))
    \end{align*}    

    rearranging, we obtain the statement of the lemma.
\end{proof}

\section{Proofs for Section \ref{sec:algorithms}}\label{app:sec3}
\textbf{Proof of Lemma \ref{lem:mfmp}}
\begin{proof}
From the algorithm, we have
\begin{equation}
\label{eqn:MPstep4}
\left\langle F(z_{k+\frac{1}{2}}) , z_{k+1} - z \right\rangle  \leq L \left\langle H(z_k) - H(z_{k+1}),  z_{k+1} - z  \right\rangle.
\end{equation}
Additionally from the algorithm we have
\begin{equation}
\label{eqn:MPstep3}
\left\langle F(z_k) , z_{k+\frac{1}{2}} - z_{k+1} \right\rangle  \leq L \left\langle H(z_k) - H(z_{k+\frac{1}{2}}),  z_{k+\frac{1}{2}} - z_{k+1}  \right\rangle
\end{equation}
Combining the two equations and rearranging, we have
\begin{align*}
\langle F(z_{k+\frac{1}{2}}),z_{k+\frac{1}{2} } -z\rangle &\leq  L \left\langle H(z_k) - H(z_{k+1}),  z_{k+1} - z  \right\rangle 
-\left\langle   F(z_k) - F(z_{k + \frac{1}{2} }) , z_{k+\frac{1}{2}} - z_{k+1} \right\rangle \nonumber 
\\
&\quad \quad + L \left\langle H(z_k) - H(z_{k+\frac{1}{2}}),  z_{k+\frac{1}{2}} - z_{k+1}  \right\rangle.\qedhere
\end{align*}
\end{proof}

\textbf{Proof of Theorem \ref{thm:1}}
\begin{proof}
Let $\Delta = LH-F$ be the difference operator. For any $K \geq k \geq 1$, it follows that, by expressing the Lemma~\ref{lem:mfmp} in terms of $\Delta$,
    \begin{align*}
    \langle F(z_{k+\frac{1}{2}}),z_{k+\frac{1}{2} } -z\rangle \nonumber 
    &\leq   \left\langle \Delta(z_k) - \Delta(z_{k+\frac{1}{2}}),  z_{k+\frac{1}{2}} - z_{k+1}  \right\rangle + L \left\langle H(z_k) - H(z_{k+1}),  z_{k+1} - z  \right\rangle
    \end{align*}
    Using the three-point property of the GBD with respect to operators $\Delta$ and $H$, respectively, we obtain
    \begin{align*}
    \langle F(z_{k+\frac{1}{2}}),z_{k+\frac{1}{2} } -z\rangle \nonumber 
    &\leq \omega_{\Delta}(z_{k+1},z_k)-\omega_\Delta(z_{k+1},z_{k+\frac{1}{2}})
    +\oint_{P_{z_kz_{k+\frac{1}{2}}z_{k+1}}} \Delta^T(r) dr -\omega_{\Delta}(z_{k+\frac{1}{2}},z_k)\\
    &\quad\quad+L(\omega_H(z,z_k)-\omega_H(z,z_{k+1})-\omega_H(z_{k+1},z_{k}))\\
    &\quad\quad+L\oint_{P_{z_kz_{k+\frac{1}{2}}z_{k+1}}} H^T(r) dr\\
    \end{align*}
     Noting that $\omega_F,\omega_H$ and $\omega_\Delta$ are positive and $\omega_\Delta \leq L\omega_H$, we obtain
    \begin{align*}
    \langle F(z_{k+\frac{1}{2}}),z_{k+\frac{1}{2} } -z\rangle \nonumber
    &\leq \oint_{P_{z_kz_{k+\frac{1}{2}}z_{k+1}}} \Delta^T(r) dr
    +L\oint_{P_{z_kz_{k+\frac{1}{2}}z_{k+1}}} H^T(r) dr\\
    &\quad\quad+L(\omega_H(z,z_k)-\omega_H(z,z_{k+1}))
    \end{align*}
     We now sum both sides of the above equation over the iterations $k$. Using the $\delta_1$ conservative and $\delta_2$ co-conservative properties of operators $F$ and $H$, we have
    \begin{align*}
    \sum_{k=1}^K \langle F(z_{k+\frac{1}{2}}),z_{k+\frac{1}{2} } -z\rangle \nonumber 
    &\leq \sum_{k=1}^K L\big(\omega_H(z,z_k)-\omega_H(z,z_{k+1}))\\
         &\quad \quad+K \delta_2+3K \delta_1
    \end{align*}
     Observing that the terms involving the general point $z$ on the right hand-side telescope, we obtain
    \begin{align*}
    \sum_{k=1}^K \langle F(z_{k+\frac{1}{2}}),z_{k+\frac{1}{2} } -z\rangle \nonumber 
    \leq K \delta_2+3KL \delta_1
    +L\big(\omega_H(z,z_1)-\omega_H(z,z_{K+1})\big)
    \end{align*}
    Dividing both sides by $K$, we obtain
    \begin{align}\label{def:prefinal1}
    \frac{1}{K}\sum_{k=1}^K \langle F(z_{k+\frac{1}{2}}),z_{k+\frac{1}{2} } -z\rangle 
    \leq  \delta_2+3L\delta_1  
    +\frac{L}{K}\big(\omega_H(z,z_1)-\omega_H(z,z_{K+1})\big) 
    \end{align}
    
    Finally, from the monotonicity of the operator $F$ and Eq.~\eqref{def:prefinal1} we have that the output $z_{out}$ of the MFMP algorithm satisfies
    \begin{align*}
        \ang{F(z),z_{out}-z} = \ang{F(z),\frac{\sum_{k=1}^K z_{k+\frac{1}{2}}}{K}-z}
        &\leq \ang{F(z_{k+\frac{1}{2}}),\frac{\sum_{k=1}^K z_{k+\frac{1}{2}}}{K}-z}\\
        &= \frac{\sum_{k=1}^K \ang{F(z_{k+\frac{1}{2}}),z_{k+\frac{1}{2}}-z}}{K}\\
        &\leq \frac{L}{K} \omega_H(z,z_1) +\delta_2+3L\delta_1,
    \end{align*}
which is the statement of the theorem.
\end{proof}

\textbf{Proof of Theorem \ref{thm:2}.}
\begin{proof}
     Applying the relative Lipschitzness property of $F$ with respect to $H$ in Lemma \ref{lem:mfmp} and applying the three-point property to the $H$ operator we obtain,

    \begin{align*}
    \langle F(z_{k+\frac{1}{2}}),z_{k+\frac{1}{2} } -z\rangle &\leq L\omega_H(z,z_k) - L\omega_H(z,z_{k+1}) 
    +L\omega_H(z_{k+1}, z_k) - L\omega_H(z_{k+1},z_k) \\
    &\quad\quad-L\omega_H(z_{k+1}, z_{k+\frac{1}{2}}) +L(\omega_H(z_{k+\frac{1}{2}},z_k)+\omega_H(z_{k+1},z_{k+\frac{1}{2}}))\\
    &\quad\quad+L\oint_{P_{z_{k+1}z_kz_{k+\frac{1}{2}}}}H^\top(r)dr-L\oint_{P_{z_{k+1}z_kz}}H^\top(r)dr- L\omega_H(z_{k+\frac{1}{2}}, z_k)
    \end{align*}

    Using the $\delta$-conservative property of $H$ and cancelling terms out we have,
 
    \begin{align*}
    \langle F(z_{k+\frac{1}{2}}),z_{k+\frac{1}{2} } -z\rangle \leq  &L\omega_H(z,z_k) - L\omega_H(z,z_{k+1}) 
    - L\omega_H(z_{k+1}, z_{k+\frac{1}{2}}) - L\omega_H(z_{k+\frac{1}{2}}, z_k)\\
    &\quad\quad+L(\omega_H(z_{k+\frac{1}{2}},z_k)+\omega_H(z_{k+1},z_{k+\frac{1}{2}}))
    +2L\delta
    \end{align*}
    
     Summing over $k$ and resolving the telescoping sequence, we have
    \begin{align*}
    \sum_{k=1}^K \langle F(z_{k+\frac{1}{2}}),z_{k+\frac{1}{2} } -z\rangle &\leq  
        L\big(\omega_H(z,z_1)-\omega_H(z,z_{K+1})\big)- L\omega_H(z_{k+1}, z_{k+\frac{1}{2}}) - L\omega_H(z_{k+\frac{1}{2}}, z_k)\\
        &\quad\quad+2KL\delta.
    \end{align*}  
    Dividing both sides by $K$, we obtain
    \begin{align*}
    \frac{1}{K} \sum_{k=1}^K \langle F(z_{k+\frac{1}{2}}),z_{k+\frac{1}{2} } -z\rangle &\leq L\big(\omega_H(z,z_1)-\omega_H(z,z_{K+1})\big)- \frac{L}{K}\big(\omega_H(z_{k+1}, z_{k+\frac{1}{2}}) + \omega_H(z_{k+\frac{1}{2}}, z_k)\big)\\
    &\quad\quad+2L\delta.
    \end{align*}   

    Furthermore, from the monotonicity of the operator $F$ we have that the output $z_{out}$ of the MFMP algorithm satisfies
    \begin{align*}
        \ang{F(z),z_{out}-z} &= \ang{F(z),\frac{\sum_{k=1}^K z_{k+\frac{1}{2}}}{K}-z}\\
        &\leq \ang{F(z_{k+\frac{1}{2}}),\frac{\sum_{k=1}^K z_{k+\frac{1}{2}}}{K}-z}\\
        &= \frac{\sum_{k=1}^K \ang{F(z_{k+\frac{1}{2}}),z_{k+\frac{1}{2}}-z}}{K}\\
        &\leq \frac{L}{K} \omega_H(z,z_1) +2L\delta,
    \end{align*}

    i.e, it approximately solves the VI objective up to error $\epsilon = O(K^{-1} + 2L\delta)$. Which is the statement of the theorem.
\end{proof}

\textbf{Proof of Theorem \ref{thm:MFMP-SM}}
\begin{proof}
From the update steps of the algorithm, at any step $k$ we have,
\begin{align*}
    \ang{F(z_k),z_{k+\frac{1}{2}}-z} &\leq \ang{L(H(z_k)-H(z_{k+\frac{1}{2}})),z_{k+\frac{1}{2}}-z}\\
    \ang{F(z_{k+\frac{1}{2}}),z_{k+1}-z} &\leq L\ang{H(z_k)-H(z_{k+1}),z_{k+1}-z}+m\ang{H(z_{k+\frac{1}{2}})-H(z_{k+1}),z_{k+1}-z}
\end{align*}
Substituting $z = z_{k+1}$ and $z = z^*$ respectively,
\begin{align*}
    \ang{F(z_k),z_{k+\frac{1}{2}}-z_{k+1}} &\leq \ang{L(H(z_k)-H(z_{k+\frac{1}{2}})),z_{k+\frac{1}{2}}-z_{k+1}}\\
    \ang{F(z_{k+\frac{1}{2}}),z_{k+1}-z^*} &\leq L\ang{H(z_k)-H(z_{k+1}),z_{k+1}-z^*}+m\ang{H(z_{k+\frac{1}{2}})-H(z_{k+1}),z_{k+1}-z^*}
\end{align*}
Adding $\ang{F(z_{k+\frac{1}{2}})-F(z_{k}),z_{k+\frac{1}{2}}-z_{k+1}}$ to each side and adding the above equations we obtain,
\begin{align}
    \ang{F(z_{k+\frac{1}{2}}),z_{k+\frac{1}{2}}-z^*} 
    &\leq \ang{L(H(z_k)-H(z_{k+\frac{1}{2}})),z_{k+\frac{1}{2}}-z_{k+1}} + L\ang{H(z_k)-H(z_{k+1}),z_{k+\frac{1}{2}}-z^*} \nonumber \\
    &\quad\quad + m\ang{H(z_{k+\frac{1}{2}})-H(z_{k+1}),z_{k+1}-z^*} 
    + \ang{F(z_{k+\frac{1}{2}})-F(z_{k}),z_{k+\frac{1}{2}}-z_{k+1}} \nonumber \\
    &\leq L(\omega_H(z_{k+1},z_{k+\frac{1}{2}})+\omega_H(z_{k+\frac{1}{2}},z_{k})) 
    - m(\omega_H(z_{k+1},z_{k+\frac{1}{2}})+\omega_H(z^*,z_{k+1})) \label{eqn:line2} \\
    &\quad\quad - \omega_H(z^*,z_{k+\frac{1}{2}})) - L(\omega_H(z^*,z_{k+1})+\omega_H(z_{k+1},z_{k})-\omega_H(z^*,z_{k})) \nonumber \\
    &\quad\quad - L(\omega_H(z_{k+1},z_{k+\frac{1}{2}})+\omega_H(z_{k+\frac{1}{2}},z_{k})-\omega_H(z_{k+1},z_{k})) \nonumber \\
    &\quad\quad + L\oint_{P_{z_{k+1}z_k z_{k+\frac{1}{2}}}} H^T(r) dr 
    + L\oint_{P_{z^*z_kz_{k+\frac{1}{2}}}} H^T(r) dr \nonumber \\
    &\quad\quad + m\oint_{P_{z^*z_{k+\frac{1}{2}}z_{k+1}}} H^T(r) dr 
    - L\oint_{P_{z_kz_{k+\frac{1}{2}}z_{k+1}}} F^T(r) dr \nonumber \\
    &= m\omega_H(z^*,z_{k+\frac{1}{2}})-m\omega_H(z_{k+1},z_{k+\frac{1}{2}}) 
    + L\omega_H(z^*,z_{k})-(m+L)\omega_H(z^*,z_{k+1})+E_k \nonumber
\end{align}
where in Eq.~\eqref{eqn:line2} we use the three-point property and $E_k$ represents the loop integrals.

Thus we have,
\begin{align*}
     (\ang{F(z_{k+\frac{1}{2}}),z_{k+\frac{1}{2}}-z^*}-m\omega_H(z^*,z_{k+\frac{1}{2}})) + (m+L)\omega_H(z^*,z_{k+1}) &\leq -m\omega_H(z_{k+1},z_{k+\frac{1}{2}})+L\omega_H(z^*,z_{k})\\
    &\leq L\omega_H(z^*,z_{k})+E
\end{align*}
since $z^*$ is also a strong solution ($F$ is continuous and monotone) we have, $\ang{F(z^*),z-z^*} \geq 0~\forall z$. Thus we have,
$$\ang{F(z_{k+\frac{1}{2}}),z_{k+\frac{1}{2}}-z^*} \geq \ang{F(z_{k+\frac{1}{2}}),z_{k+\frac{1}{2}}-z^*}-\ang{F(z^*),z_{k+\frac{1}{2}}-z^*} \geq m\omega_H(z^*,z_{k+\frac{1}{2}})$$
where the second inequality follows from strong monotonicity of $F$ with respect to $H$. This gives us,
\begin{align*}
     (m+L)\omega_H(z^*,z_{k+1}) \leq (\ang{F(z_{k+\frac{1}{2}}),z_{k+\frac{1}{2}}-z^*}-m\omega_H(z^*,z_{k+\frac{1}{2}})) + (m+L)\omega_H(z^*,z_{k+1}) \leq L\omega_H(z^*,z_{k})+E
\end{align*}
dividing both sides with $(m+L)$ we obtain,
\begin{align*}
     \omega_H(z^*,z_{k+1}) \leq  \frac{L}{m+L}\omega_H(z^*,z_{k})+\frac{E_k}{m+L}
\end{align*}

which is the statement of the lemma.
\end{proof}

\section{Proofs for Section \ref{sec:thirdorder}}\label{app:sec4}
Before we begin presenting the proofs of Section~\ref{sec:thirdorder}, we present the definition of a $p^{th}$ order model operator and higher-order smoothness.

\subsection{Definitions}
\begin{definition}[Taylor Operator]
We define $\mathcal{T}_p$ as the Taylor approximation of $\Phi$ at $z_b$ centered at $z_a$, 
\begin{equation}\label{eq:tau-def}
    \mathcal{T}_p(z_a,z_b) :=  \sum_{i=0}^p \nabla^i \Phi(z_a)[z_b-z_a]^i
\end{equation}
where $z_a,z_b$ are points in $\cz$.
\end{definition}

\begin{definition}\label{pth}($p^{th}$-Order smoothness)
\begin{equation}\label{assmpt:smooth}\tag{\textsc{a}$_1$}
    \|\nabla_{z} \Phi(z_b)-\nabla_{z}\mathcal{T}_{p-1}(z_b,z_a)\|\leq \frac{L_p}{(p-1)!}\|z_b-z_a\|^{p-1},\forall z_a,z_b \in \cz
\end{equation}
\end{definition}

\textbf{Proof of Lemma \ref{lem:thirdrel}}
We now begin the proof,

\begin{proof}
    We let $d_4(h)=\frac{\|h\|^4}{4}$. From smoothness Eq.~\ref{assmpt:smooth} we have,
\begin{align*}
   \ang{u,\nabla_{z} \Phi(z_b)[u] -\nabla_{z}\mathcal{T}_2 (z_b,z_a)[u]}  &\leq \|\nabla_{z} \Phi(z_b)-\nabla_{z}\mathcal{T}_{2}(z_b,z_a)\|\|u\|^2\\
  &\leq \frac{L_3}{2}\|z_b-z_a\|^{2}\|u\|^2,\forall z_a,z_b \in \cz
\end{align*}
rearranging, we obtain,
\begin{align*}
  0 \leq \ang{u,\nabla_{z} \Phi(z_b)[u]} & \leq \ang{u,\nabla_{z}\mathcal{T}_{2}(z_b,z_a)[u]}+\frac{L_3}{2}\|z_b-z_a\|^{2}\|u\|^2\\
  &= \ang{u,\nabla_{z} \Phi(z_b)[u] +  \nabla_{z}^2 \Phi(z_b)[z_b-z_a,u]}
+\frac{L_3}{2}\|z_b-z_a\|^{2}\|u\|^2,~\forall z_a,z_b,u \in \cz
\end{align*}
Let $z_b = z_a+\tau h$,
\begin{align*}
  0 \leq\ang{u, \nabla_{z} \Phi(z_b)[u]} & \leq  \ang{u,\nabla_{z}\mathcal{T}_{2}(z_b,z_a)[u]}+\frac{L_3}{2}\tau^{2}\|u\|^2\|h\|^2\\
  &= \ang{u,\nabla_{z} \Phi(z_a) +  \tau\nabla_{z}^2 \Phi(z_a)[h,u]}+\frac{L_3\tau^{2}}{2}\|u\|^{2}\|h\|^2,~\forall z_a,z_b,u \in \cz
\end{align*}
since $\Phi$ is monotone.
This gives,
\begin{align}\label{eqn:cubic_l}
  -\frac{1}{\tau}\ang{u,\nabla_{z} \Phi(z_a)[u]}-\frac{L_3\tau}{2}\|u\|^{2}\|h\|^2 \leq \ang{u,\nabla_{z}^2 \Phi(z_a)[h,u]} ~\forall z_a,h,u \in \cz
\end{align}
for $z_b = z_a-\tau h$ we obtain,
\begin{align}\label{eqn:cubic_r}
  \frac{1}{\tau}\ang{u,\nabla_{z} \Phi(z_a)[u]}+\frac{L_3\tau}{2}\|u\|^{2}\|h\|^2 \leq \ang{u,\nabla_{z}^2 \Phi(z_a)[h,u]} ~\forall z_a,h,u \in \cz
\end{align}
Using Eq.~\eqref{eqn:cubic_r} we have,
\begin{align*}
    \ang{u,(1-\frac{1}{\tau})\nabla_{z} \Phi(z_a)[u] +\frac{M}{2}\nabla^2_hd_4[h,u]} & \leq \ang{u,\nabla_{z} \Phi(z_a) + \nabla^2_z \Phi(z_a)[h]+\frac{M}{2}\nabla^2_hd_4(h)\big) u} +\frac{L_3\tau}{2}\|u\|^{2}\|h\|^2 
\end{align*}
Hence we have,
$$\ang{u,(1-\frac{1}{\tau})\nabla_{z} \Phi(z_a)[u] +\frac{M}{2}\nabla^2_hd_4[h,u]} -\frac{L_3\tau}{2}\|h\|^2\|u\|^2 \leq \ang{u,\nabla_{z} \Phi(z_a) + \nabla^2_z \Phi(z_a)[h,u]+\frac{M}{2}\nabla^2_hd_4[h,u]}$$
which gives,
    \begin{equation}\label{eqn:thirdsm}
        (1-\frac{1}{\tau})\nabla_{z} \Phi(z_a)[u] +\frac{M-\tau L_3}{2}u\nabla^2_hd_4(h)[h,u] \preceq \nabla_{z} \Phi(z_a) + \nabla^2_z \Phi(z_a)[h]+\frac{M}{2}\nabla^2_hd_4(h)
    \end{equation}

Thus observing that in Eq.~\eqref{eqn:thirdsm}, LHS is $\nabla_h H(h)$ and RHS is $\nabla_h \Phi(h)$ we have,
$$\ang{u,\nabla_h H(h)[u]}\leq   \ang{u,\nabla_h F(h)[u]}$$
Using the other side Eq.~\eqref{eqn:cubic_l} we get,
$$\ang{u,\nabla_h H(h)[u]} \leq \ang{u,\nabla_h F(h)[u]} \leq   \frac{\tau+1}{\tau-1}\ang{u,\nabla_h H(h)[u]} \forall~u\in \cz$$
Finally we have that $H$ is monotone since it is the linear combination of a linear operator and a monotone operator $\nabla_h d_4(h)$. Using the left inequality in the preceding equation we have that $F$ is also monotone. Thus we obtain,
$$0 \leq \ang{u, \nabla_h H(h)[u]} \leq \ang{u,\nabla_h F(h)[u]} \leq   \frac{\tau+1}{\tau-1}\ang{u,\nabla_h H(h)[u]} \forall~u \in \cz.$$

    Consider the operators $\Psi_1 = F-H$ and $\Psi_2 = (\frac{\tau+1}{\tau-1})H-F$. We know from lemma \ref{lem:jacobrels} that an operator $\Psi$ is monotone if and only if the Jacobian of the operator $\Psi$ satisfies $u^\top \nabla \Psi(z) u \geq 0 ~\forall u \in \mathcal{Z}$ at each point $z \in \mathcal{Z}$. From Lemma \ref{lem:thirdrel} we have that $\Psi_1$ and $\Psi_2$ are monotone which completes the proof.
\end{proof}

\textbf{Proof of Lemma \ref{lem:conrelop}}
\begin{proof}
The higher order derivatives are as follows, $\Phi(z) = (\nabla_x \alpha +Ay,\nabla_y \beta-A^\top x), \nabla \Phi(z)[h] = \begin{bmatrix}
    \nabla^2_x \alpha & -A\\
    A & \nabla^2_y \beta
\end{bmatrix}\begin{bmatrix}
    h_x\\h_y
\end{bmatrix}$ and $\nabla^2 \Phi(z) [h,h] = \nabla_x^3 \alpha(x) [h_x,h_x]+ \nabla_y^3 \beta(y) [h_y,h_y]$
From Lemma \ref{lem:thirdrel} we have that the operator $F = \Phi(z_a)+\nabla \Phi(z_a)^\top h+ \nabla^2 \Phi(z_a) [h,h]+\frac{M}{2}\nabla_h d_4(h)$ is relatively smooth and strongly monotone with respect to $H(h) = \frac{1}{2}(1-\frac{1}{\tau}) \nabla_z^\top \Phi(z_a)h+\frac{M-\tau L_3}{2} \nabla_h d_4(h)$. Consider the operator $H' = H-\frac{(1-\frac{1}{\tau})}{2}Nh$ where $N = \begin{bmatrix}
    0 & -A\\
    A & 0 
\end{bmatrix}$. We have that $u^\top H'u = u^\top H u~\forall u$ since $u^\top M u = 0~\forall u$. Thus we have that $\Psi_1 = F-H'$ and $\Psi_2 = (\frac{\tau+1}{\tau-1})H'-F$ satisfy $u^\top \nabla \Phi_1 u\geq 0~\forall~u$, $u^\top \nabla \Phi_2 u\geq 0~\forall~u$ and are thus monotone. Now we show that $H'$ is a conservative operator. In order to show this, it is sufficient to find a function, $\phi$ such that $\nabla \phi = H'$. We claim that $$\phi(h) = \ang{\Phi(z_a),h}+(\frac{1-\frac{1}{\tau}}{2})(\nabla_z \Phi(z_a)-N)[h,h]$$ is such a function. To show this we observe, $$\nabla_h \phi(h) = \Phi(z_a)+\frac{1-\frac{1}{\tau}}{4}\ang{(\nabla_z \Phi(z_a)-N)+(\nabla_z \Phi(z_a)-N)^\top,h}$$

Since $\nabla_z \Phi(z_a)-N = \begin{bmatrix}
    \nabla_x^2 \alpha&0\\0&\nabla_y^2\beta
\end{bmatrix}$ is a symmetric matrix we have that,
\begin{align*}
    \nabla_h \phi(h) = \Phi(z_a)+\frac{1-\frac{1}{\tau}}{2}(\nabla_z \Phi(z_a)-N)h = H'
\end{align*}
\end{proof}

\subsection{Computing the mirror-free updates}\label{app:updates}
\subsubsection{Computing \textbf{$\textrm{Prox}_H(z_k,z_k)$}}

The update is,
\begin{align*}
    \textrm{Prox}_H(z_k,z_k) &\in \{z':\ang{F(z_k)+L(H(z')-H(z_k)),z'-z}\leq0~\forall~z\}
\end{align*}
$F(h) = \Phi(z_k)+\nabla \Phi(z_k) [h]+\frac{1}{2} \nabla^2 \Phi(z_k)[h,h]+2M\nabla_h d_4(h)$
and, 
$H(h) = \frac{1}{2}(1-\frac{1}{\tau}) \nabla_z\Phi(z_k)[h]+\frac{M-\tau L_3}{2} \nabla_h d_4(h)$.
For the unconstrained, setting we have that it is equivalent to,
\begin{align}
    &\Phi(z_k)+\nabla \Phi(z_k) [z_k]+\frac{1}{2} \nabla^2 \Phi(z_k)[z_k,z_k]+2M\nabla_h d_4(z_k)\\
    &\quad\quad+L(\frac{1}{2}(1-\frac{1}{\tau}) \nabla_z\Phi(z_k)[z'-z_k]+(\frac{M-\tau L_3}{2}) (\nabla_h d_4(z') - \nabla_h d_4(z_k)))=0\nonumber
\end{align}
Rearranging, we obtain
\begin{align*}
    L((1-\frac{1}{\tau}) \nabla_z\Phi(z_k)[z']+(M-\tau L_3) (\nabla_h d_4(z')) &= L((1-\frac{1}{\tau}) \nabla_z\Phi(z_k)[z']+(M-\tau L_3) 4\|z'\|^2z') \\
    &=-(\Phi(z_k)+(1-L(1-\frac{1}{\tau}))\nabla \Phi(z_k) [z_k]\\
    &\quad\quad+\frac{1}{2} \nabla^2 \Phi(z_k)[z_k,z_k]+(2M+L(M-\tau L_3))\nabla_h d_4(z_k))\nonumber 
\end{align*}
this is equivalent to,
\begin{align*}
    \nabla_z\Phi(z_k)[z']+\frac{4(M-\tau L_3) (\|z'\|^2z'))}{L(1-\frac{1}{\tau}) } &= -\frac{U(z_a,z_k)}{L(1-\frac{1}{\tau}) }\nonumber 
\end{align*}
where $U(z_a,z_k) = \Phi(z_k)+(1-L(1-\frac{1}{\tau}))\nabla \Phi(z_k) [z_k]+\frac{1}{2} \nabla^2 \Phi(z_k)[z_k,z_k]+(2M+L(M-\tau L_3))\nabla_h d_4(z_k)$.

Or,
$$(\nabla_z\Phi(z_k)+\frac{4(M-\tau L_3) \|z'\|^2}{L(1-\frac{1}{\tau})})z' = -U(z_a,z_k)$$
setting $\rho = \frac{4(M-\tau L_3) \|z'\|^2}{L(1-\frac{1}{\tau})}$ gives,
$$z' = -(\nabla \Phi(z_k) +\rho \lambda I)^{-1}U(z_a,z_k)$$
where $\lambda = \|z'\|^2$.

Finally, we define $\phi(\lambda) = \psi(\lambda)-\frac{\lambda}{\rho}.$

\subsubsection{Computing \textbf{$\textrm{Prox}_H^{SM}(z_k, z_{k+\frac{1}{2}})$}}

Starting with the definition of the scaled-momentum proximal operator:
\[
\textrm{Prox}_H^{SM}(z_k, z_{k+\frac{1}{2}}) \in \{z': \langle F(z_{k+\frac{1}{2}}) + L(H(z') - H(z_k)) + m(H(z') - H(z_{k+\frac{1}{2}})), z' - z \rangle \leq 0~ \forall z \}.
\]

Using the definitions:
\[
F(h) = \Phi(z_k) + \nabla \Phi(z_k)[h] + \frac{1}{2} \nabla^2 \Phi(z_k)[h, h] + 2M \nabla_h d_4(h),
\]
\[
H(h) = \frac{1}{2}(1 - \frac{1}{\tau}) \nabla_z \Phi(z_k)[h] + \frac{M - \tau L_3}{2} \nabla_h d_4(h),
\]
the unconstrained update is equivalent to solving:
\begin{align*}
    &\Phi(z_k) + \nabla \Phi(z_k)[z_{k+\frac{1}{2}}] + \frac{1}{2} \nabla^2 \Phi(z_k)[z_{k+\frac{1}{2}}, z_{k+\frac{1}{2}}] + 2M \nabla_h d_4(z_{k+\frac{1}{2}}) + \\
    &L \Bigg( \frac{1}{2}(1 - \frac{1}{\tau}) \nabla_z \Phi(z_k)[z' - z_k] + \frac{M - \tau L_3}{2} (\nabla_h d_4(z') - \nabla_h d_4(z_k)) \Bigg) + \\
    &m \Bigg( \frac{1}{2}(1 - \frac{1}{\tau}) \nabla_z \Phi(z_k)[z' - z_{k+\frac{1}{2}}] + \frac{M - \tau L_3}{2} (\nabla_h d_4(z') - \nabla_h d_4(z_{k+\frac{1}{2}})) \Bigg) = 0.
\end{align*}

Rearranging terms gives:
\begin{align*}
    &\frac{L + m}{2} \Big( (1 - \frac{1}{\tau}) \nabla_z \Phi(z_{k})[z'] + (M - \tau L_3) \nabla_h d_4(z') \Big) = \\
    &-\Bigg( \Phi(z_{k}) + \nabla \Phi(z_{k})[z_{k+\frac{1}{2}}] + \frac{1}{2} \nabla^2 \Phi(z_{k})[z_{k+\frac{1}{2}}, z_{k+\frac{1}{2}}] + 2M \nabla_h d_4(z_{k+\frac{1}{2}}) + \\
    &L \Bigg( \frac{1}{2}(1 - \frac{1}{\tau}) \nabla_z \Phi(z_{k})[-z_{k}] + \frac{1}{2}(M - \tau L_3) (-\nabla_h d_4(z_{k})) \Bigg) + \\
    &m \Bigg( \frac{1}{2}(1 - \frac{1}{\tau}) \nabla_z \Phi(z_{k})[-z_{k+\frac{1}{2}}] + \frac{1}{2}(M - \tau L_3) (-\nabla_h d_4(z_{k+\frac{1}{2}})) \Bigg) \Bigg).
\end{align*}

Let us simplify the right-hand side. Grouping terms and factoring out constants:
\begin{align*}
    &\frac{L + m}{2} \Big( (1 - \frac{1}{\tau}) \nabla_z \Phi(z_{k})[z'] + (M - \tau L_3) \nabla_h d_4(z') \Big) = \\
    &-\Bigg( \Phi(z_{k}) + (1 - \frac{L + m}{2}(1 - \frac{1}{\tau})) \nabla \Phi(z_{k})[z_{k+\frac{1}{2}}] + \frac{1}{2} \nabla^2 \Phi(z_{k})[z_{k+\frac{1}{2}}, z_{k+\frac{1}{2}}] + \\
    &\Big( 2M + \frac{m}{2}(M - \tau L_3) \Big) \nabla_h d_4(z_{k+\frac{1}{2}})-\frac{L}{2}(M - \tau L_3) \nabla_h d_4(z_{k}) \Bigg).
\end{align*}

Normalizing by \(\frac{L + m}{2}(1 - \frac{1}{\tau})\), we write:
\[
\nabla_z \Phi(z_{k})[z'] + \frac{(M - \tau L_3)}{(1 - \frac{1}{\tau})} \nabla_h d_4(z') = -\frac{2 U(z_{k}, z_{k+\frac{1}{2}})}{(L + m)(1 - \frac{1}{\tau})},
\]
where
\begin{align*}
    U(z_{k}, z_{k+\frac{1}{2}}) &= \Phi(z_{k}) + (1 - \frac{L + m}{2}(1 - \frac{1}{\tau})) \nabla \Phi(z_{k})[z_{k+\frac{1}{2}}] + \frac{1}{2} \nabla^2 \Phi(z_{k})[z_{k+\frac{1}{2}}, z_{k+\frac{1}{2}}]\\
    &\quad\quad+\Big( 2M + \frac{m}{2}(M - \tau L_3) \Big) \nabla_h d_4(z_{k+\frac{1}{2}})-\frac{L}{2}(M - \tau L_3) \nabla_h d_4(z_{k})  
\end{align*}
Setting \(\rho = \frac{4(M - \tau L_3) \|z'\|^2}{1-\frac{1}{\tau}}\), this simplifies to:
\[
(\nabla_z \Phi(z_{k}) + \lambda I) z' = -U(z_{k}, z_{k+\frac{1}{2}}),
\]
where $\lambda = \rho \|z'\|^2$ or equivalently:
\[
z' = -(\nabla_z \Phi(z_{k}) +  \lambda I)^{-1} U(z_{k}, z_{k+\frac{1}{2}}).
\]
Finally, we define \(\phi(\lambda) = \psi(\lambda) - \frac{\lambda}{\rho}\).
To perform each of the updates, we must find the roots of equation $\phi(\lambda)=\psi(\lambda) - \frac{\lambda}{\rho}$. We note that the equation is identical to equation (4.11) in \cite{lin2022explicit}. The difference lies in the definition of $\lambda$ where we have it to be $\|z'\|^2$ instead of $\|z'\|$. However, we note that this change preserves the same structure, that is, $\phi'(\lambda)$ is non-positive, $\phi''(\lambda)$ is non-negative and $\phi(\lambda) \geq 0$ (strictly greater if $U\neq 0$). It then follows that this equation is solved super-linearly.

\textbf{Proof of Lemma \ref{lem:cgo}.}
\begin{proof}
    Consider the operator $\Delta(h) = \Phi_\alpha(h)-\Phi_0(h) = \frac{\alpha}{\eta}(\nabla_{xy} f(z_a) h_y, \nabla_{yx} f(z_a) h_x)$. For any two points $h_1 = (h_{1,x},h_{1,y})$ and $h_2 = (h_{2,x},h_{2,y})$, we have,
\begin{align*}
    \ang{\Delta(h_2)-\Delta(h_1),h_2-h_1} &= \ang{(\nabla_{xy} f(z_a) h_{2,y}, \nabla_{yx} f(z_a) h_{2,x})-(\nabla_{xy} f(z_a) h_{1,y}, \nabla_{yx} f(z_a) h_{1,x}),h_2-h_1}\\
    &=(h_{2,x}-h_{1,x})^\top \nabla_{xy} f(z_a) (h_{2,y}-h_{1,y})\\
    &\quad\quad-(h_{2,y}-h_{1,y})^\top \nabla_{yx} f(z_a) (h_{2,x}-h_{1,x})\\
    &\quad\quad+(h_{2,x}-h_{1,x})^\top \nabla_{xy} f(z_a) (h_{1,y}-h_{2,y})\\
    &\quad\quad-(h_{2,y}-h_{1,y})^\top \nabla_{yx} f(z_a) (h_{2,x}-h_{1,x})\\
    &=0
\end{align*}

    Thus we have,
    $$\ang{\Phi_0(h_2)-\Phi_0(h_1),h_2-h_1} \geq \ang{\Phi_\alpha(h_2)-\Phi_\alpha(h_1),h_2-h_1} \geq \ang{\Phi_0(h_2)-\Phi_0(h_1),h_2-h_1}$$
    and $\Phi_\alpha$ is 1-relatively smooth and 1-strongly monotone with respect to $\Phi_0$. Furthermore, since the operator $I(h) = h$ is 1-strongly monotone, we have that $\Phi_0$ which is a linear combination of constant operator and $I$ is 1-strongly monotone. 
\end{proof}
    
\section{Proofs for Section \ref{sec:exanples}}\label{app:sec5}

\textbf{Proof of relative smoothness of operators in example \ref{eg:smooth}.}
\begin{proof}
    Observe that the function $f(x)$ is the same as that in example 2.1 in \cite{lu2018relatively}. It thus similarly follows that
    \begin{align}\label{eqn:lunes}
        L \nabla^2 h(x) \geq \nabla^2 f(x) \geq m \nabla^2 h(x)
    \end{align}
    Now we have, $\frac{\nabla F(x,y)+\nabla F(x,y)^\top}{2} = \begin{bmatrix} \nabla^2 f(x) & 0\\ 0 & \nabla^2f(y)\end{bmatrix}$ and, $\frac{\nabla H(x,y)+\nabla H(x,y)^\top}{2} = \begin{bmatrix} \nabla^2 h(x) & 0\\ 0 & \nabla^2h(y)\end{bmatrix}$.

This gives,
$$ \frac{\nabla F(x,y)+\nabla F(x,y)^\top}{2} - m\frac{\nabla H(x,y)+\nabla H(x,y)^\top}{2} = \begin{bmatrix} \nabla^2 f(x)-m\nabla^2 h(y) & 0\\ 0 & \nabla^2f(y)-m\nabla^2 h(y)\end{bmatrix}$$

finally we have,

$$ \begin{bmatrix} \nabla^2 f(x)-m\nabla^2 h(y) & 0\\ 0 & \nabla^2f(y)-m\nabla^2 h(y)\end{bmatrix} \succeq 0$$

since diagonal blocks are PSD following, Eq.~\eqref{eqn:lunes}. We thus have from \ref{lem:jacobrels} that $F$ is $m$-strongly monotone with respect to $H$. 
\end{proof}

Before we present the proof of the statement in Example \ref{eg:eg2} we present a supporting lemma and its proof.
\begin{lemma}\label{lem:support}
    The points \( z^*, z_K \in \cz \) satisfy
    \begin{align*}
        \ang{F(z_K), z_K - z} 
        &\leq L \left( \omega(z^*, z_K) + \omega(z_K, z^*) \right) \\
        &\quad +\ang{F(z_K), z^* - z} + \forall z \in \cz.
    \end{align*}
\end{lemma}

\begin{proof}
    By the generalized Bregman three-point property,
    \begin{align*}
            \omega_F(z,z_K)& = \omega_F(z^*,z_K)-\omega_F(z^*,z)+\oint_{P_{zz^*z_K}}F^Tdr+\ang{F(z_k)-F(z),z^*-z}\\
                           & \leq \oint_{P_{zz^*z_K}}F^Tdr-\int_{z}^{z^*}F^Tdr+\ang{F(z_k),z^*-z}+L\omega(z^*,z_K)
    \end{align*}

    Where in the second step we expand $\omega_F(z^*,z)$ and use the fact that $\omega_F \leq L\omega$. Similarly expanding $\omega_F(z,z_K)$ we have,

    \begin{align*}
            \ang{F(z_K),z_K-z} & \leq \oint_{P_{zz^*z_K}}F^Tdr-\int_{z}^{z^*}F^Tdr-\int_{z_K}^zF^Tdr+\ang{F(z_k),z^*-z}+L\omega(z^*,z_K)\\
                               & = \int_{z^*}^{z_K}F^Tdr+\ang{F(z_K),z^*-z}+L\omega(z^*,z_K)\\
                               & = \ang{F(z^*),z_K-z^*}+\ang{F(z_K),z^*-z}+L(\omega(z^*,z_K)+\omega_H(z_K,z^*))\\
                               & \leq \ang{F(z_K),z^*-z}+L(\omega(z^*,z_K)+\omega_H(z_K,z^*))
    \end{align*}    

    rearranging, we obtain the statement of the lemma.
\end{proof}

\textbf{Proof of statement in Example \ref{eg:eg2}}

\begin{proof}

The operator $\Phi$ and its Jacobian are, $\Phi(z) = (4\|x\|^2x+Ay,4\|y\|^2y-A^\top x)$, $\nabla \Phi(z) = \begin{bmatrix}
    8xx^\top +4\|x\|^2I & A^\top\\
    -A & 8yy^\top+4\|y\|^2I
\end{bmatrix}$. Define the quantity $G = \max_{z\in \mathcal{Z}} \|F(z)\|$. By assumption of relative smoothness of operator $F$ with respect to an operator $H$, we have
\begin{align*}
L(\omega_H(z_b,z_c)+\omega_H(z_a,z_b))&\geq \omega_F(z_b,z_c)+\omega_F(z_a,z_b)\\
    &= \int_{z_c}^{z_b} F(r)^\top dr-\ang{F(z_c),z_b-z_c}+\int_{z_b}^{z_a} F(r)^\top dr-\ang{F(z_b),z_a-z_b}\\
    &= \omega_F(z_a,z_c)+\oint_{P_{z_az_bz_c}} F(r)^\top dr+\ang{F(z_c),z_a-z_c}\\
    &\quad\quad-\ang{F(z_c),z_b-z_c}-\ang{F(z_b),z_a-z_b}\\
    &= \omega_F(z_a,z_c)+\oint_{P_{z_az_bz_c}} F(r)^\top dr+\ang{F(z_c)-F(z_b),z_a-z_b}\\
    & \geq \oint_{P_{z_az_bz_c}} F(r)^\top dr+\ang{F(z_c)-F(z_b),z_a-z_b}
\end{align*}
where the final inequality follows from the monotonicity of the operator $F$ and Lemma \ref{lem:posgbd}.
\end{proof}

Let us decompose $\nabla F(z)$ as,

$$\nabla F(z) = \begin{bmatrix}
    8xx^\top +4\|x\|^2I & A^\top\\
    -A & 8yy^\top+4\|y\|^2I
\end{bmatrix} = \begin{bmatrix}
    8xx^\top +4\|x\|^2I & 0\\
    0 & 8yy^\top+4\|y\|^2I
\end{bmatrix}+N$$

where $N = \begin{bmatrix}
    0&A^\top\\
    -A & 0
\end{bmatrix}$. Then we have, $h^\top N h = 0~\forall~h$ since $N$ is anti-symmetric. Thus,

$$h^\top \nabla F(z) h = 8 (h_x^\top x)^2+8(h_y^\top y)^2+4\|x\|^2\|h_x\|^2+4\|y\|^2 \|h_y\|^2 $$

$$h^\top \nabla F(z) h \geq \|h\|^2 \min \{\|x\|^2,\|y\|^2\}$$

Furthermore the operator $\nabla d_4(h)$ is monotone since $d_4(h)$ is a convex function. Thus for any $z_a = (x_a,y_a)$ with $\|x_a\| \geq \kappa(z_a)$ and $\|y\|\geq \kappa(z_a)$ we have,

\begin{align}
    \ang{H(h_1)-H(h_2),h_1-h_2} &\geq \frac{1}{2}(1-\tau) (h_1-h_2)^\top \nabla_z F(z_a) (h_1-h_2)\\
        &\geq \frac{1}{2}(1-\tau) \|h_1-h_2\|^2 \min \{\|x_a\|^2,\|y_a\|^2\} \\\nonumber
        &\geq \kappa(z_a) \|h_1-h_2\|^2
\end{align}

i.e., $H$ is strongly monotone. Which implies,
\begin{align}\label{eqn:sm}
\omega_H(h_1,h_2) \geq \frac{1-\tau}{2}\kappa(z_a) \|h_1-h_2\|^2~\forall h_1,h_2
\end{align}

From lemma \ref{lem:support} we have,
$$\ang{F(z_k),z_k-z} \leq \ang{F(z_k),z^* -z}$$
since $\omega_H(z_a,z_b)\geq 0~\forall~z$.
Noting that $\ang{F(z^*),z-z^*} \geq 0~\forall ~z$ since $F$ is monotone we have,
$$\ang{F(z_k),z_k-z} \leq  \ang{F(z_k)-F(z^*),z^* -z}~\forall~z$$
where the last inequality holds due to Cauchy-Schwarz and the bound on the operator norm, $\|F(z_k)\|,\|F(z^*)\| \leq G$.

Substituting $z=z_k$ in the above gives,

\begin{align*}
    \ang{F(z_k),z_k-z} \leq \ang{F(z_k)-F(z^*),z^*-z_k} &\leq \|F(z_k)-F(z^*)\|\|z^*-z_k\|\\
    &\leq 2M\|z^*-z_k\|
\end{align*}

Thus Eq.~\eqref{eqn:sm} combined with the observation that $\ang{F(z_K),z_K-z} \leq 2G\|z^*-z_K\|~\forall~z$  gives,

$$\ang{F(z_K),z_K-z} \leq 2G \sqrt{\frac{\omega_H(z_k,z^*)}{\frac{1-\tau}{2}\kappa(z_a)}}.$$

Finally, from Corollary \ref{cor:final} we have,
$$\ang{F(z_K),z_K-z} \leq 2G\sqrt{\frac{\omega_H(z_k,z^*)}{\kappa(z_a)}} \leq \frac{2G}{\sqrt{\frac{1-\tau}{2}\kappa(z_a)}} \sqrt{(\frac{L}{m+L})^K\omega(z^*,z_{0})+\sum_{k=1}^k\frac{E_k}{(m+L)^k} }$$

where $E_k = L\oint_{P_{z_kz_{k+1}z_{k+\frac{1}{2}}}} F^T(r) dr$. This follows from the fact that $H$ is conservative since the function in example \ref{eg:eg2} is an instance of the functions described in lemma \ref{lem:conrelop}.

\textbf{Proof of Theorem \ref{lem:antilip}}
\begin{proof}
   From the three-point property of operator $F$ we have 
   $$\ang{F(z_c)-F(z_b),z_a-z_b} = \omega_F(z_a,z_b)+\omega_F(z_b,z_c)-\omega_F(z_a,z_c)-\oint_{P_{abc}}F^\top dr$$

    relative strong monotoncity implies,

    $$\ang{F(z_c)-F(z_b),z_a-z_b} \geq m\omega_H(z_a,z_b)+m\omega_H(z_b,z_c)-\omega_F(z_a,z_c)+\oint_{P_{abc}}F^\top dr$$
    
   if the operator is also relative-Lipschitz with respect to $H$ we have,

   $$L(\omega_H(z_a,z_b)+\omega_H(z_b,z_c))\geq \ang{F(z_c)-F(z_b),z_a-z_b}$$

   combining the previous two relations, we obtain the statement of theorem.
\end{proof}

\textbf{Proof of Corollary \ref{cor:norelip1}}

We provide the proofs for Examples \ref{eg:smooth} and \ref{eg:eg2}.

\textbf{Proof for Example \ref{eg:smooth}.}
\begin{proof}
    Consider a simplified two-dimensional version of the problem with $A=C=b=d=0$ the variables $E,B$ are scalars and the three points $z_a = (\theta,\theta),z_b = (0,0),z_c = (0,\theta)$. For this we have the operator $F = (\|Ex\|^2E^\top Ex+B y,-(\|Ey\|^2E^\top Ey+B^\top x)$ is relatively smooth with respect to $H(x,y) =  (\|x\|^2x,\|y\|^2y) =  (\nabla_x h_s(x,y),-\nabla_y h_s(x,y)) = (\nabla_x h(x),\nabla_y(h_y))$ where $h_s(x,y)=h(x)-h(y)$ and $h(x) = \frac{x^4}{4}$
    
    We have $\omega_H(z_a,z_b) = 2h(\theta)-2h(0)-\ang{H(0),z_a-z_b} = 2h(\theta) = 2\frac{\theta^4}{4}$, $\omega_H(z_b,z_c) = 2h(0)-(h(0)+h(\theta))+\ang{H(z_c),z_c} = -(\frac{\theta^4}{4})+\theta^4 = \frac{3}{4}(\theta^4)$ thus, $\omega_H(z_a,z_b)+\omega_H(z_b,z_c) = 5(\frac{\theta^4}{4})$.

    Let $F_{NC} = (y,-x)$ and $F_C = F - F_{NC}$, then $\oint_{P_{abc}}F^\top dr-\omega_F(z_a,z_c)= \oint_{P_{abc}}F_C^\top dr-\omega_{F_C}(z_a,z_c) +\oint_{P_{abc}}F_{NC}^\top dr-\omega_{F_{NC}}(z_a,z_c)$. We have $\oint_{P_{abc}}F_{C}^\top dr = 0$ since $F_C$ is conservative and, \begin{align*}
        \oint_{P_{abc}}F_{NC}^\top dr-\omega_{F_{NC}}(z_a,z_c) &= \int_{z_a}^{z_b}F_{NC}^\top dr+\int_{z_b}^{z_c}F_{NC}^\top dr+\int_{z_c}^{z_a}F_{NC}^\top dr-\int_{z_c}^{z_a}F_{NC}^\top dr-\ang{F_{NC}(z_c),z_c-z_a} \\
        &= \ang{F_{NC}(z_c),z_a-z_c} = \ang{(B\theta,0),(\theta,0) }= B\theta^2
    \end{align*}

    since $F_{NC}$ is perpendicular to the line passing through $z_a,z_b$ and $z_b,z_c$.

    $\omega_{F_C}(z_a,z_c) = 2f(\theta)-(f(\theta)+f(0))-\ang{F(z_c),z_a-z_c} = f(\theta)-\ang{(0,E^4\theta^4),(-\theta,0)} = f(\theta) = E^4\theta^4$.

    We thus have from Lemma \ref{lem:antilip}, 

    $$L-m \geq \frac{B\theta^2-E^4\theta^4}{\frac{5\theta^4}{4}} = d(\theta)$$

    since $\lim_{\theta \rightarrow 0} \frac{B\theta^2-E^4\theta^4}{\frac{5\theta^4}{4}} = \infty$ we have, $\exists \theta ~s.t.~ d(\theta)\geq L,~\forall L$. Thus there does not exists an $L$ such that the relative Lipschitzness condition Eq.~\ref{def:reloplip} is satisfied.
\end{proof}

\textbf{Proof for Example \ref{eg:eg2}.}
\begin{proof}
        Consider the case when $z_{a}=0$. We have,
        
        $$\phi(h) = \mathcal{T}_2(h)+\frac{M}{2}\nabla_h d_4(h) = F(z_a)+\nabla F(z_a)[h]+ \frac{1}{2}\nabla^2 F(z_a) [h,h]+\frac{M}{2}\nabla_h d_4(h) = (Ay,-A^\top x)$$
        
        and,

        $$H(h) = \frac{1}{2}(1-\tau) \nabla_{z} \Phi(z_a)h+\frac{M-\tau L_3}{2} \nabla_h d_4(h) = \nabla_h d_4(h)$$
        
        $H(x,y) =  (\|x\|^2x,\|y\|^2y) =  (\nabla_x h_s(x,y),-\nabla_y h_s(x,y)) = (\nabla_x h(x),\nabla_y(h_y))$ where $h_s(x,y)=h(x)-h(y)$ and $h(x) = \frac{x^4}{4}$
    
    We have $\omega_H(z_a,z_b) = 2h(\theta)-2h(0)-\ang{H(0),z_a-z_b} = 2h(\theta) = 2\frac{\theta^4}{4}$, $\omega_H(z_b,z_c) = 2h(0)-(h(0)+h(\theta))+\ang{H(z_c),z_c} = -(\frac{\theta^4}{4})+\theta^4 = \frac{3}{4}(\theta^4)$

    thus, $\omega_H(z_a,z_b)+\omega_H(z_b,z_c) = 5(\frac{\theta^4}{4})$.

    Let $F_{NC} = (y,-x)$ and $F_C = F - F_{NC}$, then $\oint_{P_{abc}}F^\top dr-\omega_F(z_a,z_c)= \oint_{P_{abc}}F_C^\top dr-\omega_{F_C}(z_a,z_c) +\oint_{P_{abc}}F_{NC}^\top dr-\omega_{F_{NC}}(z_a,z_c)$. We have $\oint_{P_{abc}}F_{C}^\top dr = 0$ since $F_C$ is conservative and, \begin{align*}
        \oint_{P_{abc}}F_{NC}^\top dr-\omega_{F_{NC}}(z_a,z_c) &= \int_{z_a}^{z_b}F_{NC}^\top dr+\int_{z_b}^{z_c}F_{NC}^\top dr+\int_{z_c}^{z_a}F_{NC}^\top dr-\int_{z_c}^{z_a}F_{NC}^\top dr-\ang{F_{NC}(z_c),z_c-z_a} \\
        &= \ang{F_{NC}(z_c),z_a-z_c} = \ang{(B\theta,0),(\theta,0) }= \theta^2
    \end{align*}

    since $F_{NC}$ is perpendicular to the line passing through $z_a,z_b$ and $z_b,z_c$.

    $\omega_{F_C}(z_a,z_c) = 2f(\theta)-(f(\theta)+f(0))-\ang{F(z_c),z_a-z_c} = f(\theta)-\ang{(0,E^4\theta^4),(-\theta,0)} = f(\theta) = E^4\theta^4$.

    We thus have from Lemma \ref{lem:antilip}, 

    $$L-m \geq \frac{\theta^2-\theta^4}{5\theta^4} = d(\theta)$$

    since $\lim_{\theta \rightarrow 0} \frac{\theta^2-\theta^4}{5\theta^4} = \infty$ we have, $\exists \theta ~s.t.~ d(\theta)\geq L,~\forall L$. Thus there exits no $L$ such that the relative Lipschitzness condition Eq.~\ref{def:reloplip} is satisfied.
\end{proof}

\subsection{Empirical observations}

In this subsection we provide empirical observations for the algorithm MFMP-SM on examples \ref{eg:smooth} and \ref{eg:eg2}.

\subsubsection{Example \ref{eg:smooth}}
\begin{figure}[ht]
    \centering
    \begin{subfigure}[b]{0.45\textwidth}
        \centering
        \includegraphics[width=\textwidth]{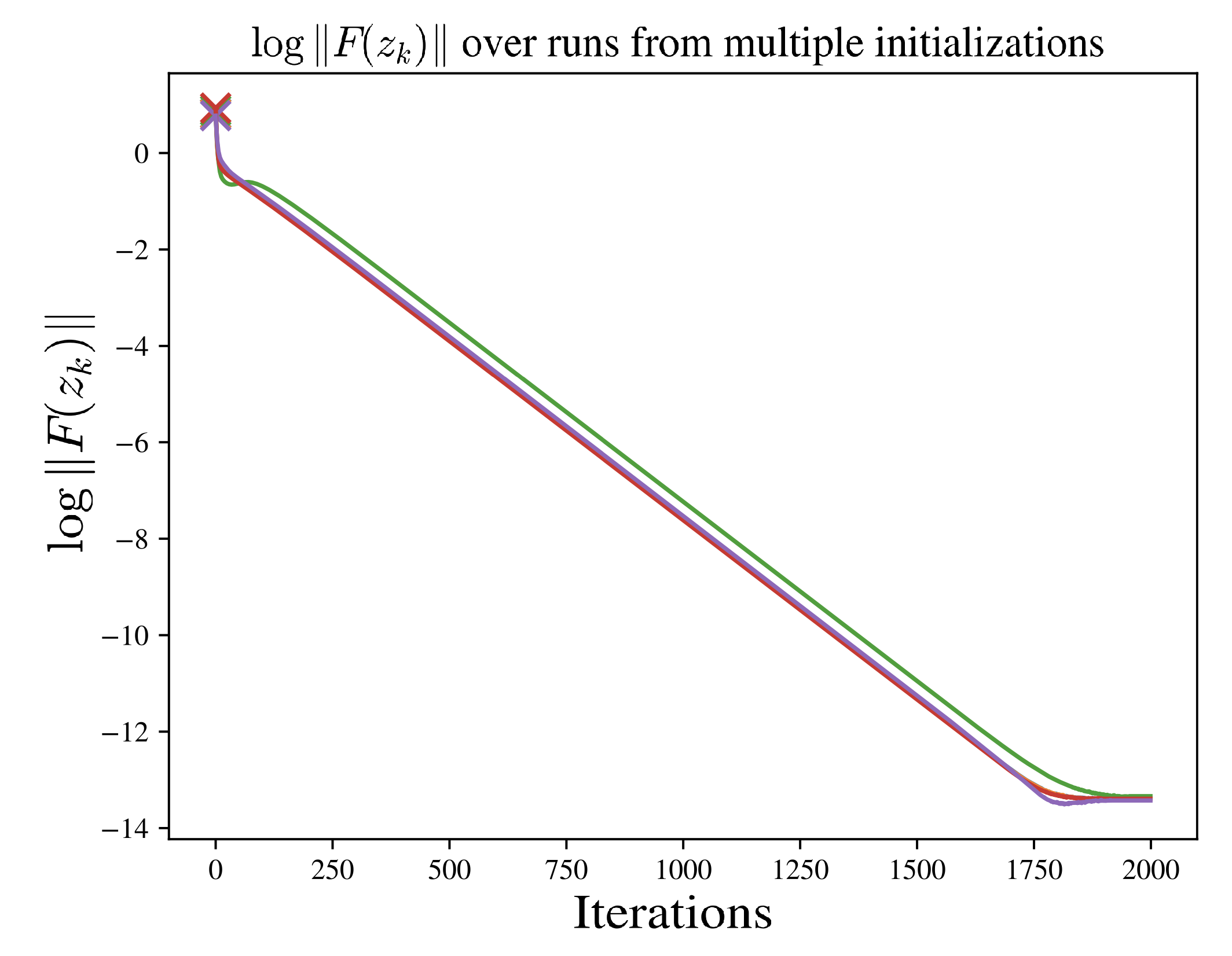}
        \caption{Inseparable case $(B\neq 0)$}
        \label{fig:1a}
    \end{subfigure}
    \hfill
    \begin{subfigure}[b]{0.45\textwidth}
        \centering
        \includegraphics[width=\textwidth]{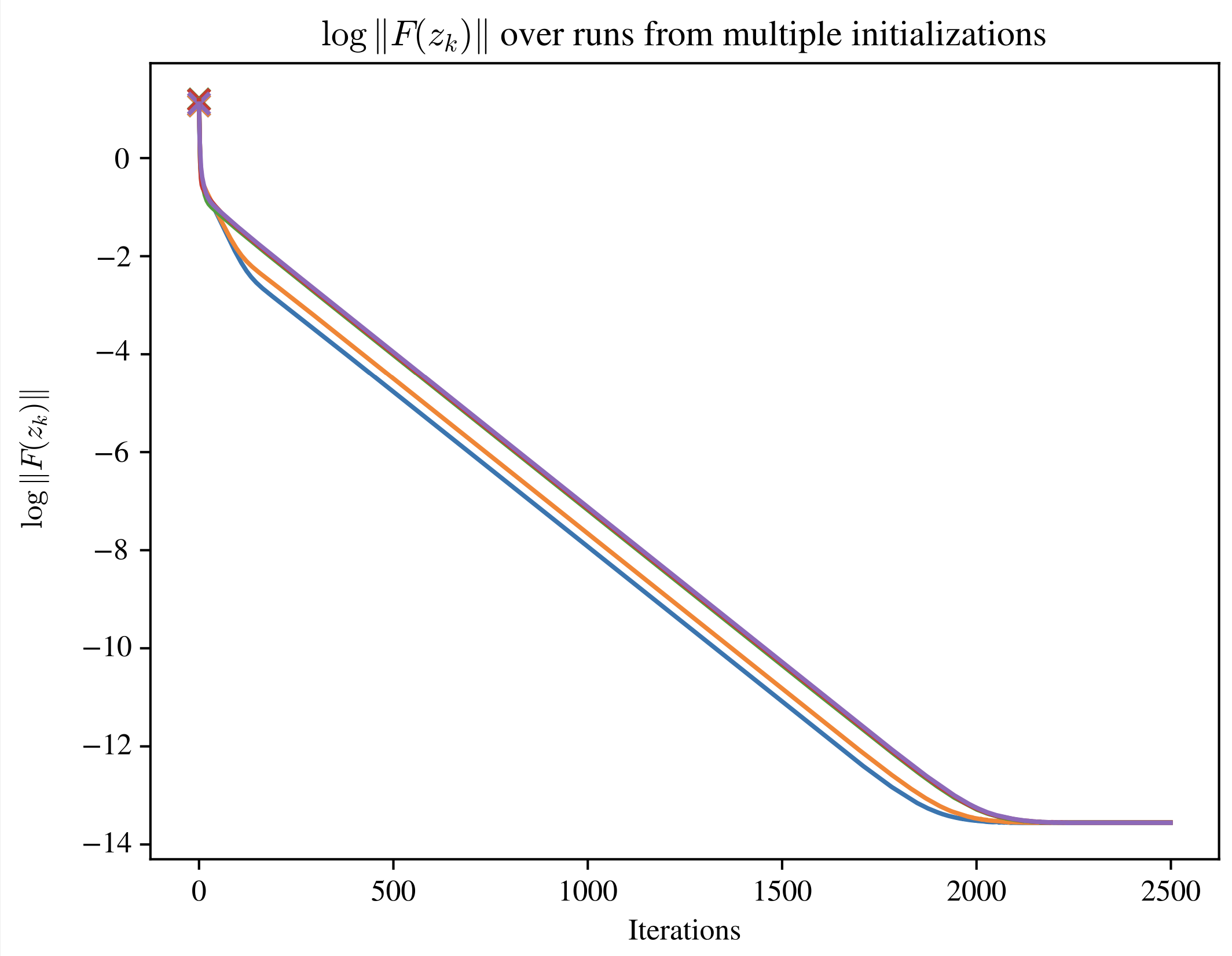}
        \caption{Separable case $(B=0)$}
        \label{fig:1b}
    \end{subfigure}
    
    \caption{MFMP-SM on Example \ref{eg:smooth} with the entries of all parameters of the function $f$ sampled from a standard normal distribution. In both cases the operator norm $\|F(z_k)\|$ converges to zero up to machine error.}
    \label{fig:main}
\end{figure}
\subsubsection{Example \ref{eg:eg2}}
\begin{figure}[ht]
    \centering
    \begin{subfigure}[b]{0.45\textwidth}
        \centering
        \includegraphics[width=\textwidth]{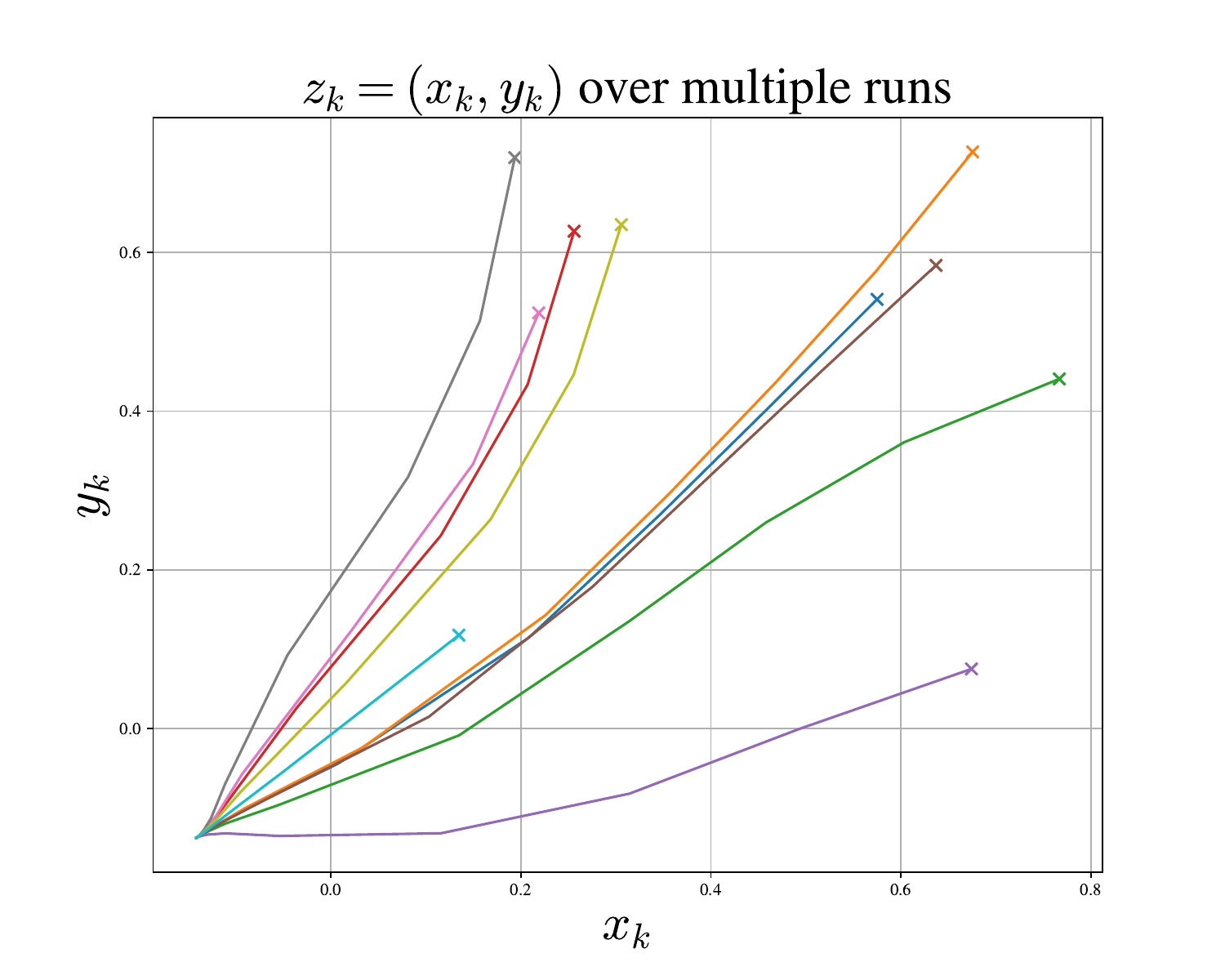}
        \caption{Iterates $z_k$ from 10 different runs. The initializations at the cross signs}
        \label{fig:2a}
    \end{subfigure}
    \hfill
    \begin{subfigure}[b]{0.45\textwidth}
        \centering
        \includegraphics[width=\textwidth]{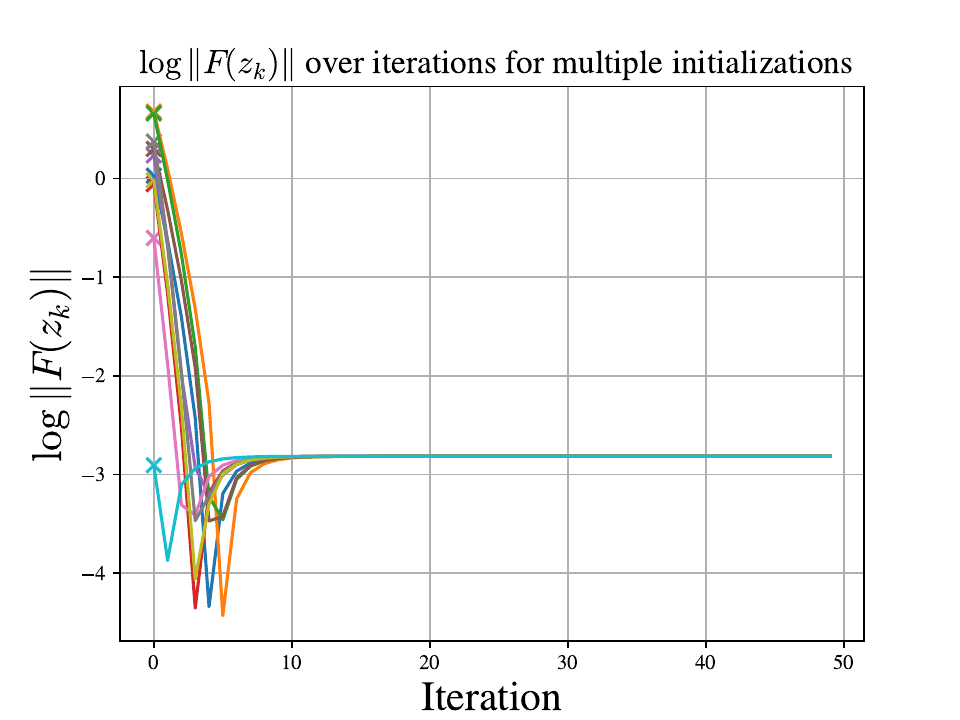}
        \caption{Convergence of $\log \|F(z_k)\|$ to non-zero error}
        \label{fig:2b}
    \end{subfigure}

    \caption{MFMP-SM on Example \ref{eg:eg2} with all parameters of the function $f$ sampled from a standard normal distribution. The algorithm was run on the sub-problem generated at $z_a$ sampled randomly from 10 different initializations.}
\end{figure}
\newpage

\end{document}